\documentclass{amsart}

\usepackage{ae,aecompl}
\usepackage{esint}
\usepackage{graphicx}
\usepackage[all,cmtip]{xy}
\usepackage{amsmath, amscd}
\usepackage{booktabs}
\usepackage{amsthm}
\usepackage{amssymb}
\usepackage{amsfonts}
\usepackage{qsymbols}
\usepackage{latexsym}
\usepackage{mathrsfs}
\usepackage{cite}
\usepackage{color}
\usepackage{url}
\usepackage{enumerate}
\usepackage{inputenc,todonotes,esint}
\usepackage{verbatim}
\usepackage[draft=false, colorlinks=true]{hyperref}
\usepackage[new]{old-arrows}

\allowdisplaybreaks

\newcommand {\A}{{\mathcal{A}}}

\newcommand {\bmo}{\mathrm{bmo}}

\newcommand {\C}{{\mathbb C}}
\newcommand {\Ca}{\mathcal{C}}

\newcommand {\Da}{\mathcal{D}}

\newcommand {\ud}{\mathrm{d}}

\newcommand {\veps}{\varepsilon}

\newcommand {\F}{\mathcal{F}}

\newcommand {\HT}{\mathcal{H}}
\newcommand {\Hp}{\mathcal{H}^{p}_{FIO}(\Rn)}

\newcommand {\Hps}{\mathcal{H}^{s,p}_{FIO}(\Rn)}

\newcommand {\ind}{\mathbf{1}}

\newcommand {\J}{\mathcal{J}}

\newcommand {\la}{\lambda}
\newcommand {\rb}{\rangle}
\newcommand {\lb}{\langle}
\newcommand {\La}{\mathcal{L}}
\newcommand {\loc}{\mathrm{loc}}

\newcommand {\Ma}{\mathcal{M}}
\newcommand {\N}{{\mathbb N}}

\newcommand {\ph}{\varphi}

\newcommand {\R}{\mathbb R}

\newcommand {\Rn}{\mathbb{R}^{n}}

\newcommand {\supp}{\mathrm{supp}}

\newcommand {\Sp}{S^{*}\Rn}
\newcommand {\Spp}{S^{*}_{+}\Rn}
\newcommand {\St}{{\mathrm{St}}}
\newcommand {\Sw}{\mathcal{S}}

\newcommand {\w}{\omega}
\newcommand {\W}{\mathrm{W}}

\newcommand {\wh}{\widehat}
\newcommand {\wt}{\widetilde}

\newcommand {\Z}{\mathbb Z}

\newcommand {\vanish}[1]{\relax}

\DeclareFontFamily{U}{mathx}{\hyphenchar\font45}
\DeclareFontShape{U}{mathx}{m}{n}{
      <5> <6> <7> <8> <9> <10>
      <10.95> <12> <14.4> <17.28> <20.74> <24.88>
      mathx10
      }{}
\DeclareSymbolFont{mathx}{U}{mathx}{m}{n}
\DeclareFontSubstitution{U}{mathx}{m}{n}
\DeclareMathAccent{\widecheck}{0}{mathx}{"71}

\DeclareMathOperator{\Real}{Re}

\newtheorem{theorem}{Theorem}[section]
\newtheorem{lemma}[theorem]{Lemma}
\newtheorem{proposition}[theorem]{Proposition}
\newtheorem{corollary}[theorem]{Corollary}

\theoremstyle{definition}
\newtheorem{definition}[theorem]{Definition}
\newtheorem{remark}[theorem]{Remark}

\numberwithin{equation}{section}
\protected\def\ignorethis#1\endignorethis{}
\let\endignorethis\relax

\title[The Hardy spaces $\Hp$ for FIOs for $p<1$]{The Hardy spaces $\Hp$ for Fourier integral operators for $p<1$}

\author{Naijia Liu}
\address{Department of Mathematics,
    Sun Yat-sen University,
    Guangzhou, 510275,
    P.R.~China}
\email{liunj@mail2.sysu.edu.cn}

\author{Jan Rozendaal}
\address{Institute of Mathematics, Polish Academy of Sciences\\
Ul.~\'{S}niadeckich 8\\
00-656 Warsaw\\
Poland}
\email{jrozendaal@impan.pl}

\author{Liang Song}
\address{Department of Mathematics,
    Sun Yat-sen University,
    Guangzhou, 510275,
    P.R.~China}
\email{songl@mail.sysu.edu.cn}



\thanks{This research was funded in part by the National Science Center, Poland, grant 2021/43/D/ST1/00667. N.~ J. Liu is supported by China Postdoctoral Science Foundation (No. 2024M763732) and
NNSF of China (No. 12501132). J. Rozendaal is partially supported by NCN grant UMO-2023/49/B/ST1/01961. L. Song is  supported by  NNSF of China (Nos. 12471097 and 12371105).}

\begin{document}

\begin{abstract}
We introduce the Hardy spaces $\Hp$ for Fourier integral operators for $0<p<1$, thereby extending earlier constructions for $1\leq p\leq \infty$. We then establish various properties of these spaces, including their behavior under complex interpolation and duality, and their invariance under Fourier integral operators. We also obtain Sobolev embeddings, equivalent characterizations, and a molecular decomposition. These spaces are used in the companion article \cite{LiRoSo25a} to determine the sharp $\HT^{1}(\Rn)$ and $\bmo(\Rn)$ regularity of wave equations with rough coefficients.
\end{abstract}
	
\maketitle

\section{Introduction}\label{sec:intro}	

In this article, we extend the definition of the Hardy spaces for Fourier integral operators to the natural full range of exponents, and we derive foundational properties of the resulting spaces.

\subsection{Setting}

The Hardy space $\HT^{1}_{FIO}(\Rn)$ for Fourier integral operators was introduced by Smith in \cite{Smith98a}, and his construction was extended in \cite{HaPoRo20} to a scale of spaces $\Hp$, for $1\leq p\leq \infty$. These spaces can be viewed as versions of the classical Lebesgue and Hardy spaces adapted to Fourier integral operators (FIOs).

Indeed, while $L^{p}(\Rn)$ is invariant under singular integral and pseudodifferential operators for $1<p<\infty$, it is not invariant under general Fourier integral operators unless $p=2$. Instead, for each $s\in\R$, a Fourier integral operator of order zero maps the Sobolev space $W^{s,p}(\Rn):=(1-\Delta)^{-s/2}L^{p}(\Rn)$ to $W^{s-2s(p),p}(\Rn)$, where 
\[
s(p):=\frac{n-1}{2}\Big|\frac{1}{p}-\frac{1}{2}\Big|,
\]
and the exponent $2s(p)$ cannot be improved. The analogous mapping property holds for $p=1$ and $p=\infty$, upon replacing $L^{p}(\Rn)$ by the local Hardy space $\HT^{1}(\Rn)$ and by $\bmo(\Rn)$, respectively. This was first shown by Peral \cite{Peral80} and Miyachi \cite{Miyachi80a} in the case of the Euclidean wave propagators, and then by Seeger, Sogge and Stein \cite{SeSoSt91} for general compactly supported Fourier integral operators of order zero, associated with a local canonical graph. As a result, they determined the sharp $L^{p}$ regularity of wave equations on compact manifolds, the solutions to which are given by Fourier integral operators.

On the other hand, Fourier integral operators of order zero do map $\Hps:=(1-\Delta)^{-s/2}\Hp$ into itself, for all $1\leq p\leq \infty$ and $s\in\R$. Moreover, the Sobolev embeddings
\begin{equation}\label{eq:Sobolevintro}
W^{s+s(p),p}(\Rn)\subseteq \Hps\subseteq W^{s-s(p),p}(\Rn)
\end{equation}
hold if $1<p<\infty$, as do appropriate analogues involving $\HT^{1}(\Rn)$ and $\bmo(\Rn)$ for $p=1$ and $p=\infty$. In particular, \eqref{eq:Sobolevintro} allows one to recover the sharp $L^{p}(\Rn)$ mapping properties of Fourier integral operators mentioned above. However, since the exponents in \eqref{eq:Sobolevintro} are sharp, one in fact obtains stronger regularity properties for Fourier integral operators when working with $\Hp$ than with $L^{p}(\Rn)$. 

The Hardy spaces for Fourier integral operators have other convenient properties, and their behavior under complex interpolation and duality mirrors that of the Lebesgue spaces. Moreover, just as the classical Hardy spaces, they can be characterized in various ways, using e.g.~maximal functions and square functions, and $\HT^{1}_{FIO}(\Rn)$ has suitable molecular and atomic decompositions. In recent years, these spaces have been applied to rough wave equations \cite{Hassell-Rozendaal23}, local smoothing  \cite{Rozendaal22b,LiRoSoYa24}, nonlinear wave equations \cite{Rozendaal-Schippa23}, and spherical maximal functions \cite{GhLiRoSo26}.

\subsection{Motivation for the present work}

In \cite{Hassell-Rozendaal23}, the sharp $L^{p}(\Rn)$ regularity of wave equations with rough coefficients was determined, thereby extending the results in \cite{SeSoSt91} for smooth coefficients. The Hardy spaces for Fourier integral operators play a crucial role in this extension, combined with techniques that were introduced in \cite{Smith98b} (see also \cite{Tataru00,Tataru01,Tataru02}) to treat rough wave equations on $L^{2}(\Rn)$.

More precisely, since the solutions to rough wave equations are not given by Fourier integral operators in the classical sense,  one cannot directly rely on standard tools from microlocal analysis to study them. Instead, \cite{Smith98b} first constructed a microlocal parametrix for the solution to the rough equation, and then used an iterative procedure to get rid of the error term.  Now, just as is the case for Fourier integral operators, this parametrix and its error terms do not leave $L^{p}(\Rn)$ invariant unless $p=2$, and therefore the iterative procedure appears to break down on $L^{p}(\Rn)$ for $p\neq 2$. To deal with this issue, \cite{Hassell-Rozendaal23} instead constructed the parametrix on $\Hp$, after which one can use the iterative procedure to remove error terms and construct a bona fide solution on $\Hp$. Finally, the embeddings in \eqref{eq:Sobolevintro} can then be used to obtain the sharp $L^{p}(\Rn)$ mapping properties of the solutions to rough wave equations.

However, \cite{Hassell-Rozendaal23} only concerns $\Hp$ and $L^{p}(\Rn)$ for $1<p<\infty$, and in particular it does not deal with the regularity properties of rough wave equations on $\HT^{1}(\Rn)$ and $\bmo(\Rn)$. This limitation arises from another aspect of the techniques which were introduced in \cite{Smith98b}. Namely, to show that the error bounds for the solution to a rough wave equation improve when successively approximating with the microlocal parametrix, one has to prove mapping properties for pseudodifferential operators with rough coefficients. Such estimates are classical on $L^{2}(\Rn)$, and even on $L^{p}(\Rn)$ for $p\neq 2$, but for the setup in \cite{Hassell-Rozendaal23} bounds on $\Hp$ were required. 

In \cite{Rozendaal23a,Rozendaal22}, bounds for rough pseudodifferential operators on $\Hp$ were first obtained, and the results in \cite{Hassell-Rozendaal23} rely crucially on those bounds. Unfortunately, the techniques in \cite{Rozendaal23a,Rozendaal22} appear to break down at the endpoints $p=1$ and $p=\infty$. 

Recently, it was realized that one can nonetheless rely on the setup in \cite{Hassell-Rozendaal23} to determine the sharp regularity of rough wave equations on $\HT^{1}(\Rn)$ and $\bmo(\Rn)$, not by directly extending the techniques in \cite{Rozendaal23a,Rozendaal22} to the endpoints but by using a more indirect approach. Namely, one can first extend the definition of $\Hp$ from $1\leq p\leq\infty$ to all $0<p\leq \infty$. Then one can prove estimates for rough pseudodifferential operators on $\Hp$ for $0<p\leq 1$ that are significantly weaker than those obtained in \cite{Rozendaal23a,Rozendaal22} for $1<p<\infty$. However, an interpolation procedure for rough symbols from \cite{Rozendaal22} can be used to combine the weaker estimates as $p\downarrow 0$ with the stronger estimates as $p\downarrow 1$, yielding the expected outcome at $p=1$. And duality then deals with the other endpoint, $p=\infty$. 

The present article constitutes one half of this approach. Namely, we extend the definition of $\Hp$ from $1\leq p\leq\infty$ to all $0<p\leq \infty$, and determine the main properties of the resulting spaces. These properties are used in the companion paper \cite{LiRoSo25a} to prove new bounds for rough pseudodifferential operators and determine the sharp $\HT^{1}(\Rn)$ and $\bmo(\Rn)$ regularity of wave equations with rough coefficients.

\subsection{Main results and techniques}

We introduce $\Hp$ for $0<p<1$ using wave packet transforms and tent spaces $T^{p}(\Sp)$ over the cosphere bundle $\Sp=\Rn\times S^{n-1}$, as in \cite{HaPoRo20}. This allows us to conveniently deduce various properties of $\Hp$ from those of $T^{p}(\Sp)$, and it links the spaces to the parametrix 
in \cite{Hassell-Rozendaal23}. In fact, one can connect the Sobolev spaces $\Hps$ over $\Hp$ to weighted tent spaces $T^{p}_{s}(\Sp)$ in the same manner. Since $\Sp$ is a doubling metric measure space when endowed with a suitable metric, basic properties of $T^{p}_{s}(\Sp)$ are available in the literature. In this manner we can show that $\Hps$ is a quasi-Banach space of tempered distributions in which the Schwartz functions are dense (see Section \ref{subsec:defHpFIO}), and that it has a suitable molecular decomposition, cf.~Theorem \ref{thm:moldecomp}. 

For all $0<p\leq \infty$ and $s\in\R$, the space $\Hps$ is invariant under suitable 
Fourier integral operators, by Theorem \ref{thm:FIObdd}. In fact, as in \cite{HaPoRo20}, for Fourier integral operators in a suitable standard form, one can remove the typical assumption that the Schwartz kernel of the operator is compactly supported. Moreover, we work with a larger class of symbols than in \cite{SeSoSt91}. 

The Sobolev embeddings in \eqref{eq:Sobolevintro} extend to $p<1$ upon replacing $W^{s,p}(\Rn)$ by a Sobolev space over the local Hardy space $\HT^{p}(\Rn)$, as is shown in Theorem \ref{thm:Sobolev}. 

Proposition \ref{prop:HpFIOdual} describes the dual of $\Hps$, also for $0<p<1$, and in Theorem \ref{thm:equivchar} we give an equivalent characterization of $\Hps$ using parabolic frequency localizations, analogous to the one in \cite{FaLiRoSo23} for $p=1$. This characterization is simpler, and more useful for the proof of the initial estimates for rough pseudodifferential operators in \cite{LiRoSo25a}. 

On the other hand, new difficulties arise for $p<1$. For example, as indicated above, it is crucial for the approach to rough wave equations on $\HT^{1}(\Rn)$ and $\bmo(\Rn)$ in \cite{LiRoSo25a} that the Hardy spaces for Fourier integral operators form a complex interpolation scale, as this allows one to interpolate analytic families of operators. However, the theory of complex interpolation does not extend directly from Banach spaces to quasi-Banach spaces. As such, for the proof of Proposition \ref{prop:HpFIOint}, which says that our spaces indeed form a complex interpolation scale, we have to establish certain geometric properties of the quasi-Banach spaces $\Hps$ and $T^{p}_{s}(\Sp)$. We have placed much of this material in Appendix \ref{sec:inter}.

There are also difficulties, both technical and fundamental, concerning low-frequency contributions. These clearly have a smoothing effect, but boundedness on the relevant spaces can still be problematic, due to the lack of $L^{p}(\Rn)$ integrability for small $p$. For example, the results on $\HT^{p}(\Rn)$ regularity in \cite{LiRoSo25a} cannot hold for general $0<p<1$ if one replaces the wave equation $\partial_{t}^{2}u(t)=Lu(t)$ by the associated half-wave equation $\partial_{t}u(t)=i\sqrt{-L}u(t)$, whenever the latter is well defined. Indeed, even in the case of the flat Laplacian in dimension $n=1$, due to the singularity at zero of $\xi\mapsto e^{i|\xi|}$, 
the operators $e^{it\sqrt{-\Delta}}$ do not have the same mapping properties as their wave counterparts 
(see Remark \ref{rem:Euclidwave}). 
Nonetheless, suitable modifications of the arguments from earlier work enable us to also prove the bounds we need for the low-frequency contributions.

\subsection{Organization of this article}

In Section \ref{sec:preliminaries} we collect preliminaries for the rest of the article. That is, we discuss the relevant homogeneous structure on the cosphere bundle and the associated tent spaces, and we provide some background on Fourier integral operators. We also introduce the wave packet transforms that we will use, and we include kernel estimates which guarantee that integral operators are bounded on tent spaces. 

In Section \ref{sec:HpFIO} we then extend the definition of $\Hp$ to all $0<p\leq \infty$, and we prove the fundamental properties of these spaces mentioned above. 

 Finally, Appendix \ref{sec:inter} contains some material on interpolation of quasi-Banach spaces. 

\subsection{Notation and terminology}\label{subsec:notation}

The natural numbers are $\N=\{1,2,\ldots\}$, and $\Z_{+}:=\N\cup\{0\}$. Throughout, we fix an $n\in\N$ with $n\geq 2$. Our construction also applies for $n=1$, but in this case one has $\HT^{p}_{FIO}(\R)=\HT^{p}(\R)$ (see Theorem \ref{thm:Sobolev}) and the results are classical. 

For $\xi\in\Rn$ we write $\lb\xi\rb=(1+|\xi|^{2})^{1/2}$, and $\hat{\xi}=\xi/|\xi|$ if $\xi\neq0$. We use multi-index notation, where $\partial_{\xi}=(\partial_{\xi_{1}},\ldots,\partial_{\xi_{n}})$ and $\partial^{\alpha}_{\xi}=\partial^{\alpha_{1}}_{\xi_{1}}\ldots\partial^{\alpha_{n}}_{\xi_{n}}$
for $\xi=(\xi_{1},\ldots,\xi_{n})\in\Rn$ and $\alpha=(\alpha_{1},\ldots,\alpha_{n})\in\Z_{+}^{n}$. Moreover, 
$\partial_{x\eta}^{2}\Phi$ is the mixed Hessian of a function $\Phi$ of the variables $x$ and $\eta$.

The bilinear duality between a Schwartz function $g\in\Sw(\Rn)$ and a tempered distribution $f\in\Sw'(\Rn)$ is denoted by $\lb f,g\rb_{\Rn}$. 
The Fourier transform of $f$ is denoted by $\F f$ or $\widehat{f}$, and the Fourier multiplier with symbol $\ph$ is denoted by $\ph(D)$. 

The measure of a measurable subset $B$ of a measure space $\Omega$ will be denoted by $|B|$, and its indicator function by $\ind_{B}$.  
For an integrable $F:B\to\C$, we write
\[
\fint_{B}F(x)\ud x=\frac{1}{|B|}\int_{B}F(x)\ud x
\]
if $|B|<\infty$. The H\"{o}lder conjugate of $p\in[1,\infty]$ is denoted by $p'$. 

The quasi-Banach space of continuous linear operators between quasi-Banach spaces $X$ and $Y$ is $\La(X,Y)$, 
and $\La(X):=\La(X,X)$. 

We write $f(s)\lesssim g(s)$ to indicate that $f(s)\leq C g(s)$ for all $s$ and a constant $C \geq0$ independent of $s$, and similarly for $f(s)\gtrsim g(s)$ and $g(s)\eqsim f(s)$.

\section{Preliminaries}\label{sec:preliminaries}
	
In this section we collect various preliminary definitions and results. 

\subsection{A metric on the cosphere bundle}\label{subsec:metric}

Here we introduce the metric measure space that will be used to define the Hardy spaces for Fourier integral operators. 
For more on this material, 
see \cite[Section 2.1]{HaPoRo20}.

We identify the cotangent bundle $T^{*}\Rn$ of $\Rn$ with $\Rn\times\Rn$,  and we write 
\[
o:=\Rn\times\{0\}\subseteq T^{*}\Rn
\]
for the zero section. 
Elements of the sphere $S^{n-1}$ in $\Rn$ are denoted by $\w$ or $\nu$, and $S^{n-1}$ is endowed with the unit normalized measure $\ud\w$. 
Let $\Sp:=\Rn\times S^{n-1}$ be the cosphere bundle of $\mathbb{R}^{n}$, endowed with the measure $\ud x\ud\w$. 
The $1$-form $\alpha_{S^{n-1}}:=\w\cdot dx$ on $\Sp$ 
gives rise to the following sub-Riemannian metric $d$ on $\Sp$:
\[
d((x,\omega),(y,v)):={\rm \inf\limits_{\gamma}}\int_{0}^{1}|\gamma^{\prime}(s)|\ud s,
\]
for $(x,\omega),(y,\nu)\in \Sp$. Here the infimum is taken over all piecewise Lipschitz 
curves $\gamma:[0,1]\rightarrow \Sp$ such that $\gamma(0)=(x,\omega)$, $\gamma(1)=(y,\nu)$ and $\alpha_{S^{n-1}}(\gamma^{\prime}(s))=0$ for almost all $s\in[0,1]$. Moreover, $|\gamma^{\prime}(s)|$ is the length of the vector $\gamma^{\prime}(s)$ with respect to the standard Euclidean metric on $\Sp$. 

By \cite[Lemma 2.1]{HaPoRo20}, 
\begin{equation}\label{eq:equivmetric}
d((x,\omega),(y,\nu))\eqsim \big{(}|\omega\cdot(x-y)|+|x-y|^{2}+|\omega-\nu|^{2}\big{)}^{1/2}
\end{equation}
for implicit constants independent of $(x,\w),(y,\nu)\in\Sp$, and we will almost exclusively work with this concrete equivalent expression for the metric.

\begin{remark}\label{rem:constantmetric}
In fact, it follows from the proof of \cite[Lemma 2.1]{HaPoRo20} that 
\[
\sqrt{\tfrac{2}{3}}\big{(}|\omega\cdot (x-y)|+|x-y|^{2}+|\omega-\nu|^{2}\big{)}^{1/2}\leq d((x,\omega),(y,\nu))
\]
for all $(x,\w),(y,\nu)\in\Spp$. 
\end{remark}

We write $B_{\tau}(x,\w)$ for the open ball around $(x,\w)\in\Sp$ of radius $\tau>0$ with respect to the metric $d$. The following statement is \cite[Lemma 2.3]{HaPoRo20}.

\begin{lemma}\label{lem:doubling}
There exists a $C>0$ such that, for all $(x,\omega)\in \Sp$, one has
\[
\frac{1}{C}\tau^{2n}\leq |B_{\tau}(x,\omega)|\leq C\tau^{2n}
\]
for $0<\tau<1$, and
\[
\frac{1}{C}\tau^{n}\leq |B_{\tau}(x,\omega)|\leq C\tau^{n}
\]
for $\tau\geq 1$. In particular,
\begin{align}\label{eq:doublingmeasure}
|B_{\lambda \tau}(x,\omega)|\leq C^{2}\lambda^{2n}|B_{\tau}(x,\omega)|
\end{align}
for all $\tau>0$ and $\lambda \geq 1$, and $(\Sp,d,\ud x\ud\omega)$ is a doubling metric measure space.
\end{lemma}


Since $(\Sp,d,\ud x\ud\w)$ is a doubling metric measure space, it follows from e.g.~\cite[Theorem 1.5]{Tozoni04}, and from the boundedness of the scalar-valued Hardy--Littlewood maximal operator, that the (centered) {vector-valued} Hardy--Littlewood maximal operator $\Ma$ is bounded on $L^{p}(\Sp;L^{q}(0,\infty))$ for all $p,q\in(1,\infty)$, where $(0,\infty)$ is endowed with the Haar measure $\frac{\ud\sigma}{\sigma}$. We will apply this to the maximal operator $\Ma_{\la}$ of index $\lambda>0$, given by 
\begin{equation}\label{eq:maxHL}
\mathcal{M}_{\lambda}g(x,\omega):=\big(\Ma(|g|^\lambda)(x,\omega)\big)^{1/{\lambda}}
\end{equation}
for $g\in L^{\lambda}_{\loc}(\Sp)$ and $(x,\w)\in\Sp$. 


Finally, for later use, we recall the following special case of a classical inequality: for every $M>n$, there exists a $C_{M}\geq0$ such that 
\begin{align}\label{eq:HLcontrol}
&{\sigma}^{-n}\int_{\Sp}(1+{\sigma}^{-1}d((x,\omega),(y,\nu))^{2})^{-M}
|g(y,\nu)|\ud y\ud\nu \nonumber\\
&\quad\leq \frac{1}{C} {|B_{\sqrt{\sigma}}(x,\omega)|}^{-1}\int_{\Sp}(1+{\sigma}^{-1/2}d((x,\omega),(y,\nu)))^{-2M}
|g(y,\nu)|\ud y\ud\nu \\
&\quad \leq C_{M}\Ma(g)(x,\omega) \nonumber
\end{align}
for all $g\in L^{1}_{\loc}(\Sp)$, $(x,\w)\in\Sp$ and $\sigma>0$. Here we used Lemma \ref{lem:doubling} and, implicitly, a  standard decomposition into dyadic annuli. In the latter argument, we need the decay exponent $2M$ to be  strictly greater than the volume doubling exponent $2n$ in \eqref{eq:doublingmeasure}. See also, for example, \cite[p.~195]{FaLiRoSo23}.

\subsection{Tent spaces}\label{subsec:tent}

In this subsection we collect some background on the tent spaces $T^{p}(\Sp)$ from \cite{HaPoRo20}, and their weighted cousins from \cite{Hassell-Rozendaal23}. Whereas in previous work these spaces were only used for $1\leq p\leq\infty$, in this article we need to consider all $0<p\leq \infty$. By Lemma \ref{lem:doubling} and with a bit of extra effort, we can nonetheless rely on the theory of tent spaces on doubling metric measure spaces as developed in e.g.~\cite{CoMeSt85, Amenta14,Amenta18}, after a straightforward change of variables to incorporate the parabolic scaling we use (see also \cite{AuKrMoPo12}).

Throughout, let 
\[
\Spp:=\Sp\times(0,\infty),
\]
endowed with the measure $\ud x\ud\w\frac{\ud\sigma}{\sigma}$. For $U\subseteq \Sp$ open, set
\[
T(U):=\{(x,\w,\sigma)\in\Spp\mid d((x,\w),U^{c})\geq \sqrt{\sigma}\}.
\]
For $F:\Spp\to \C$ measurable, $s\in\R$ and $(x,\w)\in\Sp$, write
\begin{equation}\label{eq:As}
\A_{s} F(x,\w):=\Big(\int_{0}^{\infty}\fint_{B_{\sqrt{\sigma}}(x,\w)}|F(y,\nu,\sigma)|^{2}\ud y\ud \nu\frac{\ud \sigma}{\sigma^{1+2s}}\Big)^{1/2}
\end{equation}
and, for $\alpha\in\R$,
\begin{equation}\label{eq:Cs}
\mathcal{C}_{s,\alpha}F(x,\w):=\sup_{B}\Big(\frac{1}{|B|^{1+2\alpha}}\int_{T(B)}|F(y,\nu,\sigma)|^{2}\ud y\ud \nu\frac{\ud \sigma}{\sigma^{1+2s}}\Big)^{1/2},
\end{equation}
where the supremum is taken over all balls $B\subseteq \Sp$ containing $(x,\w)$.  Also set $\Ca_{s}:=\Ca_{s,0}$.

\begin{definition}\label{def:tentspace}
For $0<p<\infty$ and $s\in\R$, the \emph{tent space} $T^{p}_{s}(\Sp)$ consists of all measurable $F:\Spp\to\C$ such that $\A_{s} F\in L^{p}(\Sp)$, endowed with the quasi-norm
\[
\|F\|_{T^{p}_{s}(\Sp)}:=\|\A_{s} F\|_{L^{p}(\Sp)}.
\]
Also, for $\alpha\in\R$, the space $T^{\infty}_{s,\alpha}(\Sp)$ consists of those measurable $F:\Spp\to\C$ such that $\mathcal{C}_{s,\alpha}F\in L^{\infty}(\Sp)$, with 
\[
\|F\|_{T^{\infty}_{s,\alpha}(\Sp)}:=\|\mathcal{C}_{s,\alpha}F\|_{L^{\infty}(\Sp)}.
\]
We write $T^{p}(\Sp):=T^{p}_{0}(\Sp)$ for $0<p<\infty$. Moreover, $T^{\infty}_{s}(\Sp):=T^{\infty}_{s,0}(\Sp)$ and $T^{\infty}(\Sp):=T^{\infty}_{0}(\Sp)$. 
\end{definition}

\begin{remark}\label{rem:tentweight}
In \cite{Amenta18}, weighted tent spaces were considered on doubling metric measure spaces, with a weight which, in our setting, would correspond to replacing $\sigma^{2s}$ in \eqref{eq:As} and \eqref{eq:Cs} by a suitable power of $|B_{\sqrt{\sigma}}(x,\w)|$. We have chosen the present definition because it is more convenient when considering the Sobolev spaces $\Hps$ over $\Hp$ (see Proposition \ref{prop:HpFIOtent}). On the other hand, due to Lemma \ref{lem:doubling}, there is no real difference between these weights whenever $F(\cdot,\cdot,\sigma)=0$ for $\sigma>e$, which is the case that will be considered later on.
\end{remark}

For all $0<p\leq \infty$ and $s\in\R$, the weighted tent space $T^{p}_{s}(\Sp)$ is a quasi-Banach space, and it is a Banach space if $p\geq 1$. Similarly, $T^{\infty}_{s,\alpha}(\Sp)$ is a Banach space for all $\alpha\in\R$. Moreover, $T^{2}(\Sp)=L^{2}(\Spp)$ isometrically, and $T^{p}(\Sp)\cap T^{2}(\Sp)$ is dense in $T^{p}(\Sp)$ for all $p\in(0,\infty)$. These assertions are a consequence of \cite[Proposition 1.4]{Amenta18} when $s=0$. Then the statement for general $s\in\R$ follows, by noting that the map sending $F\in T^{p}(\Sp)$ to
\[
(x,\w,\sigma)\to \sigma^{s}F(x,\w,\sigma)
\]
is an isometric isomorphism between $T^{p}(\Sp)$ and $T^{p}_{s}(\Sp)$, and also between $T^{\infty}_{0,\alpha}(\Sp)$ and $T^{\infty}_{s,\alpha}(\Sp)$.

In Section \ref{subsec:charac} we will want to compare the conical square function in \eqref{eq:As} to its vertical analogue. For all $0<p\leq 2$, by \cite[Proposition 2.1 and Remark 2.2]{AuHoMa12},
\begin{equation}\label{eq:vertical}
\Big(\int_{\Sp}\Big(\int_{0}^{\infty}|F(x,\w,\sigma)|^{2}\frac{\ud\sigma}{\sigma}\Big)^{p/2}\ud x\ud\w\Big)^{1/p}\lesssim \|F\|_{T^{p}(\Sp)},
\end{equation}
for an implicit constant independent of $F\in T^{p}(\Sp)$. The reverse inequality does not hold for $p<2$. On the other hand, the reverse inequality does hold for $2\leq p<\infty$, while \eqref{eq:vertical} itself fails for $p>2$. 




We include a proposition on interpolation of weighted tent spaces that will implicitly play an important role in this article. For the range of parameters where the spaces involved are Banach spaces, the statement is contained in \cite{Amenta18}, and in the Euclidean setting it can be found for almost all the parameters in \cite{HoMaMc11}. Taking into account the subtleties of complex interpolation of quasi-Banach spaces, in Appendix \ref{sec:inter} we indicate how the proof extends to our setting.

\begin{proposition}\label{prop:tentint}
Let $p_{0},p_{1}\in(0,\infty]$ be such that $(p_{0},p_{1})\neq (\infty,\infty)$, and let $p\in(0,\infty)$, $s_{0},s_{1},s\in\R$ and $\theta\in(0,1)$ be such that $\frac{1}{p}=\frac{1-\theta}{p_{0}}+\frac{\theta}{p_{1}}$ and $s=(1-\theta)s_{0}+\theta s_{1}$. Then
\[
[T^{p_{0}}_{s_{0}}(\Sp),T^{p_{1}}_{s_{1}}(\Sp)]_{\theta}=T^{p}_{s}(\Sp),
\]
with equivalent quasi-norms.
\end{proposition}

By \cite[Proposition 1.9]{Amenta18}, for all $1\leq p<\infty$ and $s\in\R$, one has
\begin{equation}\label{eq:tentdual}
(T^{p}_{s}(\Sp))^{*}=T^{p'}_{-s}(\Sp),
\end{equation}
with equivalent norms. Here the duality pairing is given by
\begin{equation}\label{eq:pairing}
(F,G)\mapsto\lb F,G\rb_{\Spp}:=\int_{\Spp}F(x,\w,\sigma)G(x,\w,\sigma)\ud x\ud \w\frac{\ud \sigma}{\sigma}
\end{equation}
for $F\in T^{p'}_{-s}(\Sp)$ and $G\in T^{p}_{s}(\Sp)$. For $0<p<1$, by \cite[Theorem 1.11]{Amenta18},
\begin{equation}\label{eq:tentdual1}
(T^{p}_{s}(\Sp))^{*}=T^{\infty}_{-s,1/p-1}(\Sp)
\end{equation}
with equivalent quasi-norms, using the same duality pairing.


Next, we include a proposition about embeddings between weighted tent spaces with different integrability parameters.

\begin{proposition}\label{prop:tentembedding}
Let $0<p_{0}\leq p_{1}\leq \infty$ and $s,\alpha\in\R$. Then there exists a $C\geq0$ such that, for all $F\in T^{p_{0}}_{s}(\Sp)$ satisfying $F(x,\w,\sigma)=0$ for all $(x,\w,\sigma)\in\Spp$ with $\sigma>e$, one has $F\in T^{p_{1}}_{s-n(\frac{1}{p_{0}}-\frac{1}{p_{1}})}(\Sp)$ and
\[
\|F\|_{T^{p_{1}}_{s-n(\frac{1}{p_{0}}-\frac{1}{p_{1}})}(\Sp)}\leq C\|F\|_{T^{p_{0}}_{s}(\Sp)},
\]
as well as $F\in T^{\infty}_{s-n(\alpha+\frac{1}{p_{0}}),\alpha}(\Sp)$ and
\[
\|F\|_{T^{\infty}_{s-n(\alpha+\frac{1}{p_{0}}),\alpha}(\Sp)}\leq C \|F\|_{T^{p_{0}}_{s}(\Sp)}.
\]
\end{proposition}
\begin{proof}
This follows from \cite[Theorem 2.19]{Amenta18}, Lemma \ref{lem:doubling} and Remark \ref{rem:tentweight}.
\end{proof}

For $0<p\leq 1$ and $s\in\R$, a measurable function $A:\Spp\to\C$ is a \emph{$T^{p}_{s}(\Sp)$ atom} if there exists an open ball $B\subseteq\Sp$ such that $\supp(A)\subseteq T(B)$ and 
\[
\int_{\Spp}|A(x,\w,\sigma)|^{2}\ud x\ud\w\frac{\ud\sigma}{\sigma^{1+2s}}\leq |B|^{-(\frac{2}{p}-1)}.
\]
The collection of $T^{p}_{s}(\Sp)$ atoms is a uniformly bounded subset of $T^{p}_{s}(\Sp)$, and the following atomic decomposition holds.

\begin{proposition}\label{prop:atomictent}
Let $0<p\leq 1$ and $s\in\R$. Then there exists a $C>0$ such that the following holds. For all $F\in T^{p}_{s}(\Sp)$, there exists a sequence $(A_{k})_{k=1}^{\infty}$ of $T^{p}_{s}(\Sp)$ atoms, and an $(\alpha_{k})_{k=1}^{\infty}\in\ell^{p}$, such that $F=\sum_{k=1}^{\infty}\alpha_{k}A_{k}$ and
\[
\frac{1}{C}\|F\|_{T^{p}_{s}(\Sp)}\leq\Big(\sum_{k=1}^{\infty}|\alpha_{k}|^{p}\Big)^{1/p}\leq C\|F\|_{T^{p}_{s}(\Sp)}.
\]
Moreover, let $R\in\La(T^{2}_{s}(\Sp))$ be such that $\|R(A)\|_{T^{p}_{s}(\Sp)}\leq C'$ for all $T^{p}_{s}(\Sp)$ atoms $A$ and a $C'\geq0$ independent of $A$. Then $R$ has a unique bounded extension from $T^{p}_{s}(\Sp)\cap T^{2}_{s}(\Sp)$ to $T^{p}_{s}(\Sp)$.
\end{proposition}
\begin{proof}
The atomic decomposition is \cite[Proposition 1.10]{Amenta18}, which in turn follows from \cite{Russ07}, and the final statement can be shown just as in \cite[Lemma 2.8]{HaPoRo20}.
\end{proof}

\begin{remark}\label{rem:balls}
In Proposition \ref{prop:atomictent}, if $F$ is such that $F(x,\w,\sigma)=0$ for all $(x,\w,\sigma)\in\Spp$ with $\sigma>1$, then one can choose the atoms $A_{k}$ to be associated with balls of radius at most $2$. For $p<1$ and $s=0$, this follows from a minor modification of the proof of this statement for $p=1$ and $s=0$, in \cite[Theorem 3.6]{CaMcMo13} (see also \cite{Russ07}). The statement for general $s\in\R$ then follows immediately.
\end{remark}

A minor role will be played by a class of test functions on $\Spp$ and the corresponding  distributions. Let $\J(\Spp)$ consist of all $G\in L^{\infty}(\Spp)$ such that 
\begin{equation}\label{eq:classJ}
\sup_{(x,\w,\sigma)\in\Spp}(1+|x|+\max(\sigma,\sigma^{-1}))^{N} |G(x,\w,\sigma)|<\infty
\end{equation}
for all $N\geq0$, with the topology generated by the corresponding weighted $L^{\infty}$ norms. 
Let $\J'(\Spp)$ be the space of  continuous linear functionals $F:\J(\Spp)\to \C$, endowed with the topology induced by $\J(\Spp)$. We denote the duality between $G\in\J(\Spp)$ and $F\in \J'(\Spp)$ by $\lb F,G\rb_{\Spp}$. \vanish{If $G\in L^{1}_{\loc}(\Spp)$ is such that 
\[
F\mapsto \int_{\Spp}F(x,\w,\sigma)\overline{G(x,\w,\sigma)}\ud x\ud\w\frac{\ud\sigma}{\sigma}
\]
defines an element of $\Da'(\Spp)$, then we write $G\in\Da'(\Spp)$. Note in particular that 
\[
L^{1}\big(\Spp,(1+|x|+\Upsilon(\sigma)^{-1})^{-N}\ud x\ud\w\frac{\ud\sigma}{\sigma}\big)\subseteq\Da'(\Spp)
\]
for all $N\geq0$.}


\begin{lemma}\label{lem:distributions}
For all $0<p<\infty$ and $s\in\R$ one has
\begin{equation}\label{eq:distributions1}
\J(\Spp)\subseteq T^{p}_{s}(\Sp)\subseteq \J'(\Spp)
\end{equation}
continuously, where the first embedding is dense. Additionally, if $p\leq1$ then
\begin{equation}\label{eq:distributions2}
\J(\Spp)\subseteq T^{\infty}_{s,1/p -1}(\Sp)\subseteq \J'(\Spp)
\end{equation}
continuously.
\end{lemma}
In this lemma, for all $F\in T^{p}_{s}(\Sp)$ or $F\in T^{\infty}_{s,1/p -1}(\Sp)$, and for all $G\in \J(\Spp)$, we consider the pairing $\lb F,G\rb_{\Spp}$ to be given by \eqref{eq:pairing}. 

\begin{proof}
It clearly suffices to consider $s=0$. For $p\geq 1$, \eqref{eq:distributions1} and the density statement are  \cite[Lemma 2.10]{HaPoRo20}. To extend the first embedding in \eqref{eq:distributions1} and the density statement to $p\in(0,1)$, one can use the same proof.  By \eqref{eq:tentdual1}, this in turn yields the second embedding in \eqref{eq:distributions2}.

To see that the second embedding in \eqref{eq:distributions1} also holds for $0<p<1$, by Proposition \ref{prop:atomictent} it suffices to prove that there exists an $N\geq 0$ such that
\[
\Big|\int_{\Spp}\!F(x,\w,\sigma)A(x,\w,\sigma)\ud x\ud\w\frac{\ud\sigma}{\sigma}\Big|\!\lesssim \!\sup_{(y,\nu,\tau)\in\Spp}\!(1+|y|+\max(\tau,\tau^{-1}))^{N}|F(y,\nu,\tau)| 
\]
for each $T^{p}(\Sp)$ atom $A$ and each $F\in\J(\Spp)$. To this end, let $A$ be associated with a ball $B\subseteq\Sp$ of radius $r>0$. Then one can use the properties of $A$ and Lemma \ref{lem:doubling} to write
\begin{align*}
&\int_{\Spp}|F(x,\w,\sigma)A(x,\w,\sigma)|\ud x\ud\w\frac{\ud\sigma}{\sigma}=\int_{0}^{r^{2}}\int_{B}|F(x,\w,\sigma)A(x,\w,\sigma)|\ud x\ud\w\frac{\ud\sigma}{\sigma}\\
&\leq \|A\|_{L^{2}(\Spp)}\Big(\int_{0}^{r^{2}}\int_{B}|F(x,\w,\sigma)|^{2}\ud x\ud\w\frac{\ud\sigma}{\sigma}\Big)^{1/2}\\
&\leq {|B|^{1-\frac{1}{p}}}\Big(\int_{0}^{r^{2}}\min(\sigma,\sigma^{-1})^{N}\frac{\ud\sigma}{\sigma}\Big)^{1/2}\sup_{(y,\nu,\tau)\in\Spp}\max(\tau,\tau^{-1})^{N}|F(y,\nu,\tau)|\\
&\lesssim \max\big(r^{2n(1-\frac{1}{p})},r^{n(1-\frac{1}{p})}\big)\min(1,r^{2N})\sup_{(y,\nu,\tau)\in\Spp}\max(\tau,\tau^{-1})^{N}|F(y,\nu,\tau)|\\
&\lesssim \sup_{(y,\nu,\tau)\in\Spp}(1+|y|+\max(\tau,\tau^{-1}))^{N}|F(y,\nu,\tau)|,
\end{align*}
for $N\geq0$ large. Finally, \eqref{eq:tentdual1} now also yields the first embedding in \eqref{eq:distributions2}.
\end{proof}

\subsection{Fourier integral operators}\label{subsec:FIOs}

For the general theory of Fourier integral operators, and the associated notions from symplectic geometry, we refer to \cite{Hormander09,Duistermaat11,Sogge17}. Readers less familiar with this theory can consider operators with a concrete representation as in Definition \ref{def:operator} below, which already cover most cases of interest.


For $m\in\R$, let $S^{m}_{1/2,1/2,1}$ consist of all $a\in C^{\infty}(\R^{2n})$ such that 
\[
\sup_{(x,\eta)\in\R^{2n}\setminus o}\lb\eta\rb^{-m+\frac{|\alpha|}{2}-\frac{|\beta|}{2}+\gamma}|(\hat{\eta}\cdot\partial_{\eta})^{\gamma}\partial_{x}^{\beta}\partial_{\eta}^{\alpha}a(x,\eta)|<\infty
\]
for all $\alpha,\beta\in\Z_{+}^{n}$ and $\gamma\in\Z_{+}$. Note that $S^{m}_{1/2,1/2,1}$ strictly contains H\"{o}rmander's 
$S^{m}_{1,1/2}$, but that it is itself strictly contained in $S^{m}_{1/2,1/2}$.

\begin{definition}\label{def:operator}
Let $m\in\R$, $a\in S^{m}_{1/2,1/2,1}$ and $\Phi\in C^{\infty}(\R^{2n}\setminus (\Rn\times\{0\})$, and set
\begin{equation}\label{eq:oscint}
Tf(x):=\int_{\Rn}e^{i\Phi(x,\eta)}a(x,\eta)\wh{f}(\eta)\ud\eta
\end{equation}
for $f\in\Sw(\Rn)$ and $x\in \Rn$. Then $T$ is a Fourier integral operator of order $m$ and type $(1/2, 1/2, 1)$ in \emph{standard form}, associated with a \emph{global canonical graph}, if:
\begin{enumerate}
\item\label{it:phase1} $\Phi$ is real-valued and positively homogeneous of degree $1$ in the $\eta$ variable;
\item\label{it:phase2} $\sup_{(x,\eta)\in \R^{2n}\setminus o}|\partial_{x}^{\beta}\partial_{\eta}^{\alpha}\Phi(x,\hat{\eta})|<\infty$ for all $\alpha,\beta\in\Z_{+}^{n}$ with $|\alpha|+|\beta|\geq 2$;
\item\label{it:phase3} $\inf_{(x,\eta)\in \R^{2n}\setminus o}| \det \partial^2_{x \eta} \Phi (x,\eta)|>0$;
\item\label{it:phase4} For each $x\in\Rn$, $\eta\mapsto \partial_{x}\Phi(x,\eta)$ is a bijection on $\R^{n}\setminus \{0\}$.
\end{enumerate}
\end{definition}

\begin{remark}\label{rem:oscint} 
By Hadamard's global inverse function theorem (see \cite[Theorem 6.2.8]{Krantz-Parks13} {or \cite{Ruzhansky-Sugimoto15})}, condition \eqref{it:phase4} is superfluous for $n\geq3$. 
If \eqref{it:phase4} holds, then the global inverse function theorem implies that the map $(\partial_{\eta}\Phi(x,\eta),\eta)\mapsto (x,\partial_{x}\Phi(x,\eta))$ is a homogeneous canonical transformation on $\R^{2n}\setminus o$, and the canonical relation of $T$ is the graph of this transformation.
\end{remark}



\subsection{Wave packet transforms}\label{subsec:transforms}

In this subsection we introduce the wave packets and wave packet transforms that are used to define the Hardy spaces for Fourier integral operators. For more on this material, see \cite[Section 4]{HaPoRo20} and \cite[Section 3]{Rozendaal21}.

Throughout, let $\Psi\in C^{\infty}_{c}(\mathbb{R}^{n})$ be a non-negative radial function such that $\Psi(\xi)=0$ for all $\xi\in\Rn$ with $|\xi|\notin[\frac{1}{2},2]$, and
\begin{equation}\label{eq:Psi}
\int_{0}^{\infty}\Psi(\sigma\xi)^{2}\frac{\ud\sigma}{\sigma}=1
\end{equation}
if $\xi\neq 0$. Fix a non-negative radial $\varphi \in C_{c}^{\infty}(\mathbb{R}^{n})$ such that $\varphi\equiv 1$ in a small neighborhood of zero, and $\varphi(\xi)=0$ for $|\xi|>1$. 
Write $c_{\sigma}:=\big{(}\int_{S^{n-1}}\varphi(\frac{e_{1}-\nu}{\sqrt{\sigma}})^{2}\ud\nu\big{)}^{-1/2}$ for $\sigma>0$, where $e_{1}$ is the first basis vector of $\Rn$. 
For $\w\in S^{n-1}$, $\sigma>0$ and $\xi\in\Rn\setminus\{0\}$, set 
\[
\psi_{\w,\sigma}(\xi)=\Psi(\sigma\xi)c_{\sigma}\ph\big(\tfrac{\hat{\xi}-\w}{\sqrt{\sigma}}\big)
\]
and $\psi_{\w,\sigma}(0):=0$. 
Moreover, let
\[
\rho(\xi):=\Big(1-\int_{0}^{1}\Psi(\sigma \xi)^{2}\frac{\ud\sigma}{\sigma}\Big)^{1/2}
\]
for $\xi\in\Rn$. Then $\rho\in C^{\infty}_{c}(\Rn)$, with $\rho(\xi)=1$ for $|\xi|\leq 1/2$, and $\rho(\xi)=0$ if $|\xi|\geq 2$.

We also introduce parabolic cutoffs associated with these wave packets. For $\w\in S^{n-1}$ and $\xi\in\Rn$, set
\begin{equation}\label{eq:phw}
\ph_{\omega}(\xi):=\int_{0}^{4}\psi_{\w,\tau}(\xi)\frac{\ud\tau}{\tau}
\end{equation}
and
\begin{equation}\label{eq:theta}
\theta_{\omega,\sigma}(\xi):=\Psi(\sigma\xi)\varphi_{\omega}(\xi).
\end{equation}
Some properties of $(\ph_{\w})_{\w\in S^{n-1}}\subseteq C^{\infty}(\Rn)$ are as follows (see \cite[Remark 3.3]{Rozendaal21}):
\begin{enumerate}
\item\label{it:phiproperties1} For all $\w\in S^{n-1}$ and $\xi\neq0$ one has $\ph_{\w}(\xi)=0$ if $|\xi|<\frac{1}{8}$ or $|\hat{\xi}-\w|>2|\xi|^{-1/2}$.
\item\label{it:phiproperties2} For all $\alpha\in\Z_{+}^{n}$ and $\beta\in\Z_{+}$ there exists a $C_{\alpha,\beta}\geq0$ such that
\[
|(\w\cdot \partial_{\xi})^{\beta}\partial^{\alpha}_{\xi}\ph_{\w}(\xi)|\leq C_{\alpha,\beta}|\xi|^{\frac{n-1}{4}-\frac{|\alpha|}{2}-\beta}
\]
for all $\w\in S^{n-1}$ and $\xi\neq0$.
\item\label{it:phiproperties3}
There exists a radial $m\in S^{(n-1)/4}(\Rn)$ such that, for each $f\in\Sw'(\Rn)$ satisfying $\supp(\wh{f}\,)\subseteq \{\xi\in\Rn\mid |\xi|\geq1/2\}$, one has
\begin{equation}\label{eq:phiproperties3}
f=\int_{S^{n-1}}m(D)\ph_{\nu}(D)f\ud\nu.
\end{equation}
\end{enumerate}
In \eqref{it:phiproperties3}, $S^{(n-1)/4}(\Rn)$ consists of the standard symbols of order $(n-1)/4$. 

The following statement is contained in \cite[Lemma 3.2]{Rozendaal21}.

\begin{lemma}\label{lem:packetbounds}
For $\w\in S^{n-1}$ and $\sigma\in(0,1)$, let $\eta_{\w,\sigma}\in\{\psi_{\w,\sigma},\theta_{\w,\sigma}\}$. Then $\eta_{\w,\sigma}\in C^{\infty}_{c}(\Rn)$, and each $\xi\in\supp(\eta_{\w,\sigma})$ satisfies $\frac{1}{2}\sigma^{-1}\leq |\xi|\leq 2\sigma^{-1}$ and $|\hat{\xi}-\w|\leq 2\sqrt{\sigma}$. 
Moreover, for all $\alpha\in\Z_{+}^{n}$ and $\beta\in\Z_{+}$ there exists a $C_{\alpha,\beta}\geq0$, independent of $\w$ and $\sigma$, such that
\begin{equation}\label{eq:boundspsi}
|(\w\cdot\partial_{\xi})^{\beta}\partial_{\xi}^{\alpha}\eta_{\w,\sigma}(\xi)|\leq C_{\alpha,\beta}\sigma^{-\frac{n-1}{4}+\frac{|\alpha|}{2}+\beta}
\end{equation}
for all $\xi\in\Rn$. Also, for each $N\geq0$ there exists a $C_{N}\geq0$, independent of $\w$ and $\sigma$, such that
\begin{equation}\label{eq:boundspsiinverse}
|\F^{-1}(\eta_{\w,\sigma})(x)|\leq C_{N}\sigma^{-\frac{3n+1}{4}}(1+\sigma^{-1}|x|^{2}+\sigma^{-2}(\w\cdot x)^{2})^{-N}
\end{equation}
for all $x\in\Rn$.
\end{lemma}

\begin{remark}\label{rem:thetatilde}
Let $\wt{\Psi}\in C^{\infty}_{c}(\Rn)$ be such that $\wt{\Psi}(\xi)=0$ for $|\xi|\notin [\frac{1}{4},4]$, and $\wt{\Psi}\equiv 1$ on $\supp(\Psi)$. For $\w\in S^{n-1}$, $\sigma>0$ and $\xi\in\Rn$, set
\[
\tilde{\theta}_{\w,\sigma}(\xi):=\sigma^{\frac{n-1}{4}}m(\xi)\ph_{\w}(\xi)\wt{\Psi}(\sigma\xi),
\]
where $m$ is as in \eqref{eq:phiproperties3}.
Then the $\tilde{\theta}_{\w,\sigma}$ have the same properties as in Lemma \ref{lem:packetbounds}, with the exception that each $\xi\in\supp(\tilde{\theta}_{\w,\sigma})$ satisfies $\frac{1}{4}\sigma^{-1}\leq |\xi|\leq 4\sigma^{-1}$. More generally, \eqref{eq:boundspsi} and \eqref{eq:boundspsiinverse} still hold if one imposes slightly different support conditions on the $\eta_{\w,\sigma}$.
\end{remark}


We now define transforms associated with these wave packets. For $f\in \Sw'(\Rn)$ and $(x,\w,\sigma)\in\Spp$, set
\begin{equation}\label{eq:defW}
Wf(x,\w,\sigma):=\begin{cases}\psi_{\w,\sigma}(D)f(x)&\text{if }0<\sigma<1,\\
\ind_{[1,e]}(\sigma)\rho(D)f(x)&\text{if }\sigma\geq1.\end{cases}
\end{equation}
Moreover, for $G$ an element of the class $\J(\Spp)$ from \eqref{eq:classJ}, and for $x\in\Rn$, set
\[
VG(x):=\int_{0}^{1}\int_{S^{n-1}}\psi_{\nu,\tau}(D)G(\cdot,\nu,\tau)(x)\ud\nu\frac{\ud\tau}{\tau}+\int_{1}^{e}\int_{S^{n-1}}\rho(D)G(\cdot,\nu,\tau)(x)\ud \nu\frac{\ud\tau}{\tau}.
\]
The following lemma, essentially contained in \cite[Proposition 4.3]{HaPoRo20}, collects some basic properties of $W$ and $V$.

\begin{proposition}\label{prop:transforms}
The following statements hold:
\begin{enumerate}
\item\label{it:transforms1} $W:\Sw(\Rn)\to \J(\Spp)$ and $W:\Sw'(\Rn)\to \J'(\Spp)$ are continuous;
\item\label{it:transforms2} $W:L^{2}(\Rn)\to L^{2}(\Spp)$ is an isometry;
\item\label{it:transforms3} $V:\J(\Spp)\to \Sw(\Rn)$ is continuous;
\item\label{it:transforms4} $\lb f,VG\rb_{\Rn}=\lb Wf,G\rb_{\Spp}$ for all $f\in\Sw'(\Rn)$ and $G\in\J(\Spp)$. 
\end{enumerate}
\end{proposition}

We may thus extend $V$ to a continuous map $V:\J'(\Spp)\to\Sw'(\Rn)$ by setting
\[
\lb VF,g\rb_{\Rn}:=\lb F,Wg\rb_{\Spp}
\]
for $F\in\J'(\Spp)$ and $g\in\Sw(\Rn)$, and then \begin{equation}\label{eq:repro}
VWf=f
\end{equation}
for all $f\in\Sw'(\Rn)$.

Finally, we include a lemma, an extension of \cite[Lemma A.1]{HaPoRo20} to $p<1$, that will allow us to conveniently deal with the low-frequency component of $W$, given by
\begin{equation}\label{eq:defWl}
W_{l}f(x,\w,\sigma):=\ind_{[1,e]}(\sigma)\rho(D)f(x)
\end{equation}
for $f\in\Sw'(\Rn)$ and $(x,\w,\sigma)\in\Spp$. Also write
\begin{equation}\label{eq:defWh}
W_{h}f(x,\w,\sigma):=\ind_{(0,1)}(\sigma)\psi_{\w,\sigma}(D)f(x),
\end{equation}
so that $Wf=W_{l}f+W_{h}f$.

\begin{lemma}\label{lem:Wlow}
Let $0<p\leq 1$ and $s\in\R$
. Then there exists a $C>0$ such that an $f\in\Sw'(\Rn)$ satisfies $\W_{l}f\in T^{p}_{s}(\Sp)$ if and only if $\rho(D)f\in L^{p}(\Rn)$, in which case
\begin{equation}\label{eq:Wlow}
\frac{1}{C}\|W_{l}f\|_{T^{p}_{s}(\Sp)}\leq \|\rho(D)f\|_{L^{p}(\Rn)}\leq C\|W_{l}f\|_{T^{p}_{s}(\Sp)}.
\end{equation}
\end{lemma}
\begin{proof}
It suffices to consider $s=0$. Then the second inequality in \eqref{eq:Wlow} is proved as in the case where $p=1$, in \cite[Lemma A.1]{HaPoRo20}, relying also on Jensen's inequality. 

For the first inequality we argue slightly differently, due to subtleties concerning Sobolev embeddings for $p<1$. Let $R>0$ be large enough such that $B_{\sqrt{\sigma}}(x,\w)\subseteq B_{R}(x)\times S^{n-1}$ for all $(x,\w,\sigma)\in\Spp$ with $\sigma\leq e$, and let $h\in\Sw(\Rn)$ have compact Fourier support and be such that $|h(x)|\geq 1$ if $|x|\leq R$. Then
\begin{align*}
\|W_{l}f\|_{T^{p}(\Sp)}&\leq \Big(\int_{\Sp}\Big(\int_{1}^{e}\int_{S^{n-1}}\int_{B_{R}(x)}|\rho(D)f(y)|^{2}\ud y\ud\nu\frac{\ud\sigma}{\sigma}\Big)^{p/2}\ud x\ud\w\Big)^{1/p}\\
&\leq \Big(\int_{\Rn}\Big(\int_{\Rn}|h(x-y)\rho(D)f(y)|^{2}\ud y\Big)^{p/2}\ud x\Big)^{1/p}.
\end{align*}
Now note that the function $y\mapsto h(x-y)\rho(D)f(y)$ has compact Fourier support, independent of $x$. Hence \cite[Theorem 1.4.1]{Triebel10} yields
\[
\|W_{l}f\|_{T^{p}(\Sp)}\lesssim \Big(\int_{\Rn}\int_{\Rn}|h(x-y)\rho(D)f(y)|^{p}\ud x\ud y\Big)^{1/p}\lesssim \|\rho(D)f\|_{L^{p}(\Rn)}.\qedhere
\]
\end{proof}

\subsection{Operators on phase space}\label{subsec:operatorphase}

By conjugating with wave packet transforms, we will often reduce the analysis of operators on $\Rn$ to operators on $\Spp$. In this subsection we collect some results about the operators that arise in this manner. 

We consider operators $S:\J(\Spp)\to\J'(\Spp)$ given by a measurable kernel $K:\Spp\times\Spp\to \C$:
\begin{equation}\label{eq:kernelphase}
SF(x,\w,\sigma)=\int_{\Spp}K(x,\w,\sigma,y,\nu,\tau)F(y,\nu,\tau)\ud y\ud\nu\frac{\ud\tau}{\tau}
\end{equation}
for $F\in \J(\Spp)$ and $(x,\w,\sigma)\in\Spp$. Our main result for such operators is as follows.

\begin{proposition}\label{prop:offsing}
Let $0<p\leq \infty$ and $s\in\R$. Then there exists an $N\geq0$ such that the following holds. Let $S:\J(\Spp)\to\J'(\Spp)$ be as in \eqref{eq:kernelphase}, and suppose that there exist a bi-Lipschitz $\hat{\chi}:\Sp\to\Sp$ and a $C\geq0$ such that
\[
|K(x,\w,\sigma,y,\nu,\tau)|\leq C\zeta^{n}\min\big(\tfrac{\sigma}{\tau},\tfrac{\tau}{\sigma}\big)^{N}(1+\zeta d((x,\w),\hat{\chi}(y,\nu))^{2})^{-N}
\]
for all $(x,\w,\sigma),(y,\nu,\tau)\in\Spp$, where we write $\zeta:=\max(\sigma^{-1},\tau^{-1})$.
Then $S\in\La(T^{p}_{s}(\Sp))$. 
\end{proposition}
\begin{proof}
The statement is contained in \cite[Theorem 3.7]{HaPoRo20} for $p\geq1$ and $s=0$. 

For $p<1$ and $s=0$, the proof of the boundedness of $S$ is completely analogous to that given in \cite[Theorem 3.7]{HaPoRo20} for $p=1$, using the atomic decomposition of $T^{p}(\Sp)$ from Proposition \ref{prop:atomictent}. By Lemma \ref{lem:distributions}, $S$ then extends uniquely to a bounded operator on all of $T^{p}(\Sp)$. 

As in the proof of \cite[Proposition 2.4]{Hassell-Rozendaal23}, the result for $p<\infty$ and general $s\in\R$ then follows by considering the kernel 
\[
\wt{K}(x,\w,\sigma,y,\nu,\tau):=\big(\tfrac{\tau}{\sigma}\big)^{s}K(x,\w,\sigma,y,\nu,\tau),
\]
which satisfies similar bounds. Finally, for $p=\infty$, 
simply consider the adjoint action of $S$, as is allowed due to \eqref{eq:tentdual} and due to the fact that the assumption on the kernel is symmetric with respect to the variables $(x,\w,\sigma)$ and $(y,\nu,\tau)$.
\end{proof}



A specific instance to which Proposition \ref{prop:offsing} applies is the case where an operator as in Definition \ref{def:operator} is conjugated with the wave packet transforms from the previous subsection.

\begin{corollary}\label{cor:FIOtent}
Let $T$ be a Fourier integral operator of order $0$ { and type (1/2,1/2,1)} in standard form, associated with a global canonical graph, with symbol $a\in S^{0}_{1/2,1/2,1}$ and phase function $\Phi\in C^{\infty}(\R^{2n}\setminus o)$. Suppose that either $(x,\eta)\mapsto \Phi(x,\eta)$ is linear in $\eta$, or that there exists a $c>0$ such that $a(x,\eta)=0$ for all $(x,\eta)\in\R^{2n}$ with $|\eta|\leq c$. Then $WTV\in \La(T^{p}_{s}(\Sp))$ for all $0<p<\infty$ and $s\in\R$. 
\end{corollary}
\begin{proof}
By Proposition \ref{prop:transforms}, $WTV:\J(\Spp)\to\J'(\Spp)$ is a priori well defined. Hence the conclusion follows by combining Proposition \ref{prop:offsing} with \cite[Lemma 2.13 and Corollary 5.2]{HaPoRo20}.
\end{proof}

\begin{remark}\label{rem:FIOtentother}
We note that, for a given $T$, the proof of \cite[Corollary 5.2]{HaPoRo20} only makes use of the properties of the wave packets stated in Lemma \ref{lem:packetbounds}. In fact, given any collection of wave packets with these properties, one can define a transform $\wt{W}$ in the same manner as before. Then $\wt{W}TV\in \La(T^{p}_{s}(\Sp))$ in Corollary \ref{cor:FIOtent}.
\end{remark}

\section{Hardy spaces for Fourier integral operators}\label{sec:HpFIO}

In this section we introduce the Hardy spaces for Fourier integral operators $\Hp$ for all $0<p\leq \infty$, and we derive their basic properties.

\subsection{Definition and basic properties}\label{subsec:defHpFIO}

To define $\Hp$, 
we make use of the wave packet transform $W$ from \eqref{eq:defW} and the tent spaces $T^{p}(\Sp)$ from Definition \ref{def:tentspace}.

\begin{definition}\label{def:HpFIO}
Let $0<p\leq\infty$. Then $\Hp$ consists of those $f\in\Sw'(\Rn)$ such that $Wf\in T^{p}(\Sp)$, endowed with the quasi-norm
\[
\|f\|_{\Hp}:=\|Wf\|_{T^{p}(\Sp)}.
\]
Moreover, for $s\in\R$, set $\Hps:=\lb D\rb^{-s}\Hp$. 
\end{definition}

It follows from \eqref{eq:repro} that $\|\cdot\|_{\Hp}$ is indeed a quasi-norm for all $0<p\leq\infty$, and a norm for $p\geq1$.


The following proposition will allow us to reduce various properties of $\Hps$ to those of the weighted tent spaces $T^{p}_{s}(\Sp)$ from Section \ref{subsec:tent}.

\begin{proposition}\label{prop:HpFIOtent}
Let $0<p\leq \infty$ and $s\in\R$. Then there exists a $C>0$ such that the following holds. An $f\in\Sw'(\Rn)$ satisfies $f\in\Hps$ if and only if $Wf\in T^{p}_{s}(\Sp)$, and 
\begin{equation}\label{eq:HpFIOtent}
\frac{1}{C}\|Wf\|_{T^{p}_{s}(\Sp)}\leq \|f\|_{\Hps}\leq C\|Wf\|_{T^{p}_{s}(\Sp)}.
\end{equation}
Moreover, $V:T^{p}_{s}(\Sp)\to \Hps$ is bounded and surjective. 
\end{proposition}
\begin{proof}
We will follow 
the reasoning from \cite[Proposition 3.4]{Hassell-Rozendaal23}, which contains an analogous statement for $p\geq1$, albeit with a slightly different wave packet transform and using a different parametrization of $\Spp$.

For $g\in\Sw'(\Rn)$ and $(x,\w,\sigma)\in\Spp$, set
\[
\wt{W}_{1}g(x,\w,\sigma):=
\begin{cases}
\sigma^{s}\lb D\rb^{s}\psi_{\w,\sigma}(D)f(x)&\text{if }0<\sigma<1,\\
\ind_{[1,e]}(\sigma)\sigma^{s}\lb D\rb^{s}\rho(D)f(x)&\text{if }\sigma\geq1.
\end{cases}
\]
That is, $\wt{W}_{1}$ is as in \eqref{eq:defW}, but with $\psi_{\w,\sigma}(\xi)$ replaced by $\wt{\psi}_{\w,\sigma}(\xi):=\sigma^{s}\lb \xi\rb^{s}\psi_{\w,\sigma}(\xi)$,
and with $\rho(\xi)$ replaced by $\sigma^{s}\lb \xi\rb^{s}\rho(\xi)$. Note that the factor of $\sigma$ in this low-frequency contribution is bounded from above and below. Now, it is easy to check that $\wt{\psi}_{\w,\sigma}$ satisfies the same type of estimates as $\psi_{\w,\sigma}$, from Lemma \ref{lem:packetbounds}, and similarly for the low-frequency term. Hence, by Remark \ref{rem:FIOtentother} with $T$ the identity operator, $\wt{W}_{1}V\in \La(T^{p}_{s}(\Sp))$. 

Next, suppose that $Wf\in T^{p}_{s}(\Sp)$. Then, because $\lb D\rb^{s}$ commutes with $VW$, we can combine what have just shown with \eqref{eq:repro} to obtain
\begin{align*}
\|f\|_{\Hps}&=\|W\lb D\rb^{s}f\|_{T^{p}(\Sp)}=\|WVW\lb D\rb^{s}f\|_{T^{p}(\Sp)}\\
&=\|\wt{W}_{1}VWf\|_{T^{p}_{s}(\Sp)}\lesssim  \|Wf\|_{T^{p}_{s}(\Sp)},
\end{align*}
which is the second inequality in \eqref{eq:HpFIOtent}.

The argument for the first inequality in \eqref{eq:HpFIOtent} is analogous, but 
this time one replaces $\psi_{\w,\sigma}(\xi)$ 
by $\sigma^{-s}\lb \xi\rb^{-s}\psi_{\w,\sigma}(\xi)$, and 
$\rho(\xi)$ 
by $\sigma^{-s}\lb \xi\rb^{-s}\rho(\xi)$. 
Finally, the last statement follows by combining \eqref{eq:repro} and Corollary \ref{cor:FIOtent}.
\end{proof}

As in \cite{HaPoRo20}, we can now easily 
 derive various basic properties of $\Hp$ from those of $T^{p}(\Sp)$. For example, Proposition \ref{prop:HpFIOtent} 
implies that 
$W\Hps$ is a complemented subspace of $T^{p}_{s}(\Sp)$ for all $0<p\leq \infty$, and that
\[
V:T^{p}_{s}(\Sp)/\ker(V)\to \Hps
\]
is an isomorphism. In particular, $\Hps$ is a quasi-Banach space for all $s\in\R$, and a Banach space if $p\geq 1$. Moreover, as in the proof of \cite[Proposition 6.4]{HaPoRo20}, one can use Remark \ref{rem:FIOtentother} 
to show that the definition of $\Hp$ is, up to quasi-norm equivalence, independent of the particular choice of wave packets. One instance of this is the fact that $\HT^{2}_{FIO}(\Rn)=L^{2}(\Rn)$ isometrically.

Finally, as in the proof of \cite[Proposition 6.6]{HaPoRo20}, Corollary \ref{cor:FIOtent} and Lemma \ref{lem:distributions} can be combined to show that the continuous embeddings
\[
\Sw(\Rn)\subseteq \Hps\subseteq\Sw'(\Rn)
\]
hold for all $0<p\leq \infty$ and $s\in\R$, and that the first inclusion is dense if $p<\infty$. In fact, this reasoning shows that the Schwartz functions with compact Fourier support are dense in $\Hps$ if $p<\infty$. On the other hand, 
$\Sw(\Rn)$ is not dense in $\HT^{s,\infty}_{FIO}(\Rn)$ for any $s\in\R$, for example because nonzero constant functions are contained in $\HT^{s,\infty}_{FIO}(\Rn)$ but 
cannot be approximated in the $\HT^{s,\infty}_{FIO}(\Rn)$ norm by Schwartz functions (this follows e.g.~from  Theorem \ref{thm:Sobolev}).

\subsection{Interpolation and duality}\label{subsec:interdual}

A crucial role will be played in this article by the following proposition on complex interpolation of the Hardy spaces for Fourier integral operators, an extension of \cite[Proposition 6.7]{HaPoRo20} and \cite[Corollary 3.5]{Hassell-Rozendaal23}.

\begin{proposition}\label{prop:HpFIOint}
Let $p_{0},p_{1}\in(0,\infty]$ be such that $(p_{0},p_{1})\neq (\infty,\infty)$, and let $p\in(0,\infty)$, $s_{0},s_{1},s\in\R$ and $\theta\in(0,1)$ be such that $\frac{1}{p}=\frac{1-\theta}{p_{0}}+\frac{\theta}{p_{1}}$ and $s=(1-\theta)s_{0}+\theta s_{1}$. Then
\[
[\HT^{s_{0},p_{0}}_{FIO}(\Rn),\HT^{s_{1},p_{1}}_{FIO}(\Rn)]_{\theta}=\Hps.
\]
\end{proposition}
\begin{proof}
The statement for $p_{0},p_{1}\geq 1$ is contained in \cite[Corollary 3.5]{Hassell-Rozendaal23}. 
In our setting some additional care is required, due to the subtleties of complex interpolation of quasi-Banach spaces (see Appendix \ref{sec:inter}). 

As noted in the proof of Proposition \ref{prop:tentint} in Appendix \ref{sec:inter}, $T^{p_{0}}_{s_{0}}(\Sp)+T^{p_{1}}_{s_{1}}(\Sp)$ is analytically convex. 
Hence Proposition \ref{prop:HpFIOtent}, \eqref{eq:repro}, \cite[Lemma 7.11]{KaMaMi07} and Proposition \ref{prop:tentint} combine to yield the desired statement:
\begin{align*}
&[\HT^{s_{0},p_{0}}_{FIO}(\Rn),\HT^{s_{1},p_{1}}_{FIO}(\Rn)]_{\theta}=[VT_{s_{0}}^{p_{0}}(\Sp),VT_{s_{1}}^{p_{1}}(\Spp)]_{\theta}\\
&=V\big([T_{s_{0}}^{p_{0}}(\Sp),T_{s_{1}}^{p_{1}}(\Spp)]_{\theta}\big)=VT^{p}_{s}(\Sp)=\Hps.
\end{align*}
For completeness we note that, to directly apply \cite[Lemma 7.11]{KaMaMi07}, $\HT^{s_{0},p_{0}}_{FIO}(\Rn)\cap \HT^{s_{1},p_{1}}_{FIO}(\Rn)$ should be dense in both $\HT^{s_{0},p_{0}}_{FIO}(\Rn)$ and $\HT^{s_{1},p_{1}}_{FIO}(\Rn)$. However, this assumption is not used in the proof of that result.
\end{proof}

\begin{remark}\label{rem:HpFIOanconvex}
It is of independent interest to note that, under the assumptions of Proposition \ref{prop:HpFIOint}, $\HT^{s_{0},p_{0}}_{FIO}(\Rn)+\HT^{s_{1},p_{1}}_{FIO}(\Rn)$ is analytically convex. Indeed, by Proposition \ref{prop:HpFIOtent}, this space is isomorphic to a closed subspace of $T^{p_{0}}_{s_{0}}(\Sp)+T^{p_{1}}_{s_{1}}(\Sp)$. As already noted above, the latter space is analytically convex, and then \cite[Lemma 7.5]{KaMaMi07} implies that $\HT^{s_{0},p_{0}}_{FIO}(\Rn)+ \HT^{s_{1},p_{1}}_{FIO}(\Rn)$ is analytically convex as well.
\end{remark}

\begin{remark}\label{rem:interinfty}
In \cite[Corollary 3.5]{Hassell-Rozendaal23} (see also \cite[page 7]{Rozendaal22} and \cite[equation (2.3) and Remark 4.12]{LiRoSoYa24}), Proposition \ref{prop:HpFIOint} is stated for $p_{0},p_{1}\in[1,\infty]$, but without the additional condition that $(p_{0},p_{1})\neq (\infty,\infty)$. This omission is unfortunate, because the proof given there does not cover the case $p_{0}=p_{1}=\infty$, and because there are subtleties regarding the choice of complex interpolation method for spaces in which the Schwartz functions do not lie dense (see e.g.~\cite[Section 2.4.7]{Triebel10} and \cite[Section 13.A]{Taylor23c}). It is not immediately clear to us whether Proposition \ref{prop:HpFIOint} also holds in a suitable sense for $p_{0}=p_{1}=\infty$. Given that the latter case is not relevant for this article or its companion \cite{LiRoSo25a}, we will not pursue this matter further here.
\end{remark}

Next, we include a duality statement which extends \cite[Proposition 6.8]{HaPoRo20}. 

\begin{proposition}\label{prop:HpFIOdual}
Let $p\in[1,\infty)$ and $s\in\R$. Then $\Hps^{*}=\HT^{-s,p'}_{FIO}(\Rn)$. 
Moreover, for $p\in(0,1)$, the space $\Hps^{*}$ consists of those $f\in\Sw'(\Rn)$ such that $Wf\in T^{\infty}_{-s,1/p-1}(\Sp)$, and there exists a $C>0$ such that 
\[
\frac{1}{C}\|Wf\|_{T^{\infty}_{-s,1/p-1}(\Sp)}\leq \|f\|_{\Hps^{*}}\leq C\|Wf\|_{T^{\infty}_{-s,1/p-1}(\Sp)}
\]
for all such $f$.
\end{proposition}
\begin{proof}
The first statement is contained in \cite[Proposition 6.8]{HaPoRo20} for $s=0$, from which the case of general $s\in\R$ follows immediately. 
The argument for $p\in(0,1)$ is analogous, relying on Propositions \ref{prop:transforms} and \ref{eq:HpFIOtent}, the density of $\Sw(\Rn)$ in $\Hps$ and of $\J(\Spp)$ in $T^{p}(\Sp)$, and \eqref{eq:repro}, to reduce matters to 
\eqref{eq:tentdual1}. 
\end{proof}

\begin{remark}\label{rem:dualrelation}
In Proposition \ref{prop:HpFIOdual}, the duality relation is the standard pairing $\lb f,g\rb_{\Rn}$ between $f\in \HT^{-s,p'}_{FIO}(\Rn)\subseteq\Sw'(\Rn)$ and $g\in\Sw(\Rn)\subseteq\Hps$ for $p\geq1$, and analogously for $p<1$. This determines the action of $f$ uniquely, since $\Sw(\Rn)$ is dense in $\Hps$ for all $0<p<\infty$ and $s\in\R$. We will continue writing $\lb f,g\rb_{\Rn}$ for the bilinear pairing between general $f\in(\Hps)^{*}$ and $g\in\Hps$, and similarly for the sesquilinear pairing $\lb f,g\rb=\lb f,\overline{g}\rb_{\Rn}$. The former pairing is also given by the duality relation $\lb Wf,Wg\rb_{\Spp}$ from \eqref{eq:pairing}. 
%
\end{remark}

\subsection{Molecular decomposition}\label{subsec:moldecomp}

In this subsection we obtain a decomposition of $\Hp$, for $0<p\leq 1$, in terms of what were called coherent molecules in \cite{Smith98a}.

 Recall the definition of the metric $d$ on $\Sp$ from Section \ref{subsec:metric}, and the equivalent expression for it in \eqref{eq:equivmetric}.

\begin{definition}\label{def:molecule}
Let $0<p\leq 1$ and $N\geq0$. 
An $f\in L^{2}(\Rn)$ is a \emph{coherent $\Hp$ molecule} of type $N$, 
associated with a ball $B_{\sqrt{\tau}}(y,\nu)$, for $(y,\nu)\in\Sp$ and $\tau>0$, if 
\[
\supp(\wh{f}\,)\subseteq \{\xi\in \mathbb{R}^{n}\mid |\xi|\geq\tau^{-1},|\hat{\xi}-\nu|\leq \sqrt{\tau}\}
\]
and 
\begin{equation}\label{eq:molecule}
\int_{\Rn} \big(1+\tau^{-1}d((x,\nu),(y,\nu))^{2}\big)^{N}|f(x)|^{2}\ud x\leq |B_{\sqrt{\tau}}(y,\nu)|^{-(\frac{2}{p}-1)}.
\end{equation}
\end{definition}

The following theorem extends \cite[Theorem 3.6]{Smith98a} and \cite[Theorem 8.3]{HaPoRo20} to $p<1$. 

\begin{theorem}\label{thm:moldecomp}
Let $0<p\leq 1$ and $N>2n(\frac{2}{p}-1)$. Then the following assertions hold.
\begin{enumerate}
\item\label{it:moldecomp1} For every $\tau_{0}>0$ there exists a $C\geq 0$ such that the following holds. For 
each sequence $(f_{k})_{k=1}^{\infty}$ of coherent $\Hp$ molecules of type $N$, associated with balls of radius at most $\tau_{0}$, 
and each $(\alpha_{k})_{k=1}^{\infty}\in\ell^{p}$, one has $\sum_{k=1}^{\infty}\alpha_{k}f_{k}\in\Hp$ and
\[
\Big\|\sum_{k=1}^{\infty}\alpha_{k}f_{k}\Big\|_{\Hp}\leq C\Big(\sum_{k=1}^{\infty}|\alpha_{k}|^{p}\Big)^{1/p}.
\]
\item\label{it:moldecomp2} There exist $\tau_{0},C>0$ such that, for each $f\in\Hp$, there exists a sequence $(f_{k})_{k=1}^{\infty}$ of coherent $\Hp$ molecules of type $N$, associated with balls of radius at most $\tau_{0}$, 
and an $(\alpha_{k})_{k=1}^{\infty}\in \ell^{p}$, such that $f=\sum_{k=1}^{\infty}\alpha_{k}f_{k}+\rho(D)^{2}f$ and
\[
\Big(\sum_{k=1}^{\infty}|\alpha_{k}|^{p}\Big)^{1/p}\leq C\|f\|_{\Hp}.
\]
\end{enumerate}
\end{theorem}
\begin{proof}
The proof is similar to that 
of \cite[Theorem 8.3]{HaPoRo20}, which in turn is a modification of the proof of \cite[Theorem 3.6]{Smith98a}. We just make a few remarks.

\eqref{it:moldecomp1}: Since $\Hp$ is a quasi-Banach space, it suffices to show that the collection of coherent $\Hp$ molecules of type $N$, associated with balls of radius at most $\tau_{0}$, is uniformly bounded in $\Hp$. Let $f$ be such a molecule, associated with a ball $B_{\sqrt{\tau}}(y,\nu)$, for $(y,\nu)\in\Sp$ and $0<\tau\leq \tau_{0}^{2}$. We may suppose that $y=0$. With notation as in \eqref{eq:defWl} and \eqref{eq:defWh}, we will separately bound $\|W_{l}f\|_{T^{p}(\Sp)}$ and $\|W_{h}f\|_{T^{p}(\Sp)}$. 

For the first term, note that $W_{l}f=0$ unless $\tau\geq 1/2$. Suppose that the latter holds. Lemma \ref{lem:Wlow} yields $\|W_{l}f\|_{T^{p}(\Sp)}\lesssim \|\rho(D)f\|_{L^{p}(\Rn)}$. For $p=1$, one can now use that that $\rho(D)$ has an $L^{1}(\Rn)$ kernel, and then apply H\"{o}lder's inequality, \eqref{eq:equivmetric}, \eqref{eq:molecule} with $y=0$, and Lemma \ref{lem:doubling}:
\[
\|\rho(D)f\|_{L^{1}(\Rn)}\lesssim \|f\|_{L^{1}(\Rn)}\lesssim \Big(\int_{\Rn}(1+|x|^{2})^{n}|f(x)|^{2}\ud x\Big)^{1/2}\lesssim \tau^{-n/4}\lesssim 1.
\]
On the other hand, for $p<1$, first note that $\rho(D)f$ can be expressed as a convolution, since $f\in L^{2}(\Rn)$. Then Jensen's inequality yields
\begin{align*}
\|\rho(D)f\|_{L^{p}(\Rn)}&=\Big(\int_{\Rn}(1+|x|^{2})^{n}\Big|\int_{\Rn}\F^{-1}\rho(x-z)f(z)\ud z\Big|^{p}\frac{\ud x}{(1+|x|^{2})^{n}}\Big)^{1/p}\\
&\lesssim \int_{\Rn}|f(z)|\int_{\Rn}(1+|x|^{2})^{n(\frac{1}{p}-1)}|\F^{-1}\rho(x-z)|\ud x\ud z\\
&\lesssim \int_{\Rn}(1+|z|^{2})^{n(\frac{1}{p}-1)}|f(z)|\ud z,
\end{align*}
where for the last inequality we split the integral over $x$ into an integral over the region where $|x|<2|z|$, plus the remainder. Now proceed as for $p=1$.



For the term $\|W_{h}f\|_{T^{p}(\Sp)}$, we will reason as in the proof of \cite[Theorem 8.3]{HaPoRo20}. The argument relies crucially on \cite[Lemma {8.2}]{HaPoRo20}, a statement about the mapping properties of our wave packet transforms with respect to weighted $L^{2}$ spaces (see also \cite[Lemma 2.14]{Smith98a}). To apply that lemma, observe that 
\begin{equation}\label{eq:weights}
\begin{aligned}
1+\tau^{-1}d((x,\nu),(y,\nu))^{2}&\lesssim (1+\tau^{-1}|\nu\cdot(x-y)|)(1+\tau^{-1}|x-y|^{2})\\
&\lesssim (1+\tau^{-1}d((x,\nu),(y,\nu))^{2})^{2}
\end{aligned}
\end{equation}
for all $x,y\in\Rn$, $\nu\in S^{n-1}$ and $\tau>0$, by \eqref{eq:equivmetric}.

We now go into more detail. We want to prove that $\|W_{h}f\|_{T^{p}(\Sp)}\lesssim 1$. Following the argument in the proof of \cite[Theorem 8.3]{HaPoRo20}, to this end it suffices to 
show that there exist $\varepsilon,C_{\veps}>0$, independent of $f$, such that 
\begin{align}\label{eq:tentatom}
\|A_k\|_{L^2(S^*_+(\Rn))}\leq C_{\veps} 2^{-k\varepsilon} |2^{k+1}B|^{\frac{1}{2}-\frac{1}{p}}
\end{align}
for all $k\in\N$. Here $A_{1}:=\ind_{T(4B)}W_{h}f$, $A_k:={\bf 1}_{T(2^{k+1}B)\backslash T(2^kB)}W_{h}f$ for $k\geq 2$, and $B:=B_{\sqrt{\tau}}(y,\nu)$.
Also set
\[
 m^{\nu,\tau}(x):=(1+\tau^{-1}|\nu\cdot x|)(1+\tau^{-1}|x|^{2})
\]
for $x\in\Rn$.

For a given $k_{0}\geq 2$, one can combine Proposition \ref{prop:transforms} \eqref{it:transforms2}, \eqref{eq:molecule} and \eqref{eq:doublingmeasure} to obtain \eqref{eq:tentatom} for all $k<k_{0}$. On the other hand, if we choose $k_{0}$ sufficiently large, then for all $k\geq k_{0}$ we can combine \cite[Lemma {8.2}]{HaPoRo20}, \eqref{eq:weights}, \eqref{eq:molecule} and \eqref{eq:doublingmeasure}:
\begin{align*}
\|A_k\|^2_{L^2(S^*_+(\Rn))}&\lesssim 2^{-Nk}\int_{\Rn} m^{\nu,\tau}(x-y)^{N/2}|f(x)|^2\ud x\\
&\lesssim 2^{-Nk}\int_{\Rn} \big(1+\tau^{-1}d((x,\nu),(y,\nu))^{2}\big)^{N}|f(x)|^2\ud x\\
&\lesssim 2^{-Nk} |B|^{-(\frac{2}{p}-1)}\lesssim 2^{-k(N-2n(\frac{2}{p}-1))} |2^{k+1}B|^{-(\frac{2}{p}-1)}.
\end{align*}
Since $N>2n(\frac{2}{p}-1)$, this yields \eqref{eq:tentatom}.

\eqref{it:moldecomp2}: With notation as in \eqref{eq:defWl} and \eqref{eq:defWh}, one can use Proposition \ref{prop:atomictent} and Remark \ref{rem:balls} to obtain a sequence $(A_{k})_{k=1}^{\infty}$ of $T^{p}(\Sp)$ atoms, associated with balls of radius at most $2$, and an $(\alpha_{k})_{k=1}^{\infty}\subseteq\ell^{p}$, such that $W_{h}f=\sum_{k=1}^{\infty}\alpha_{k}A_{k}$ and 
\[
\Big(\sum_{k=1}^{\infty}|\alpha_{k}|^{p}\Big)^{1/p}\lesssim \|W_{h}f\|_{T^{p}(\Sp)}\leq \|Wf\|_{T^{p}(\Sp)}=\|f\|_{\Hp}.
\]
Since $VW_{l}f=\rho(D)^{2}f$, it then suffices to show that there exists a $\tau_{0}>0$, independent of $f$, such that each $VA_{k}$ is a coherent $\Hp$ molecule of type $N$, associated with a ball of radius at most $\tau_{0}$. To do so, one can proceed as in the proof of \cite[Theorem 8.3]{HaPoRo20}, again using \eqref{eq:weights}.  
\end{proof}

\begin{remark}\label{rem:moldecomp}
In fact, both \cite[Theorem 3.6]{Smith98a} and \cite[Theorem 8.3]{HaPoRo20} treat the low-frequencies in part \eqref{it:moldecomp2} slightly differently. The present formulation will be more convenient for the proof of Theorem \ref{thm:Sobolev}. We also note that \cite[Theorem 8.3]{HaPoRo20} involves an additional constant on the right-hand side of \eqref{eq:molecule}; this constant can immediately be removed, by rescaling.



Using Remark \ref{rem:constantmetric}, one can check that it suffices to let $\tau_{0}=4+\sqrt{6}$ in \eqref{it:moldecomp2}. On the other hand, this bound can be improved (see e.g.~\cite[Remark A.3]{CaMcMo13}).
\end{remark}

\subsection{Embeddings}\label{subsec:Sobolev}

In this subsection we prove embeddings involving the Hardy spaces for Fourier integral operators.


Recall that, for $0<p<\infty$, the real Hardy space $H^{p}(\Rn)$ of Fefferman and Stein consists of those $f\in\Sw'(\Rn)$ such that
\begin{equation}\label{eq:Hp}
\|f\|_{H^{p}(\Rn)}:=\Big(\int_{\Rn}\Big(\int_{0}^{\infty}|\Psi(\sigma D)f(x)|^{2}\frac{\ud\sigma}{\sigma}\Big)^{p/2}\ud x\Big)^{1/p}<\infty.
\end{equation}
Here $\Psi$ is as in \eqref{eq:Psi}. We also write $H^{\infty}(\Rn):=BMO(\Rn)=H^{1}(\Rn)^{*}$. 

Throughout, fix a $q\in C^{\infty}_{c}(\Rn)$ satisfying $q(\xi)=1$ for all $\xi\in\Rn$ with $|\xi|\leq 2$. 
For $0<p\leq \infty$, the local Hardy space $\HT^{p}(\Rn)$ consists of all $f\in\Sw'(\Rn)$ such that $q(D)f\in L^{p}(\Rn)$ and $(1-q(D))f\in H^{p}(\Rn)$, endowed with the quasi-norm
\[
\|f\|_{\HT^{p}(\Rn)}:=\|q(D)f\|_{L^{p}(\Rn)}+\|(1-q(D))f\|_{H^{p}(\Rn)}.
\]
We write $\HT^{s,p}(\Rn):=\lb D\rb^{-s}\HT^{p}(\Rn)$ for $s\in\R$.


The definition of $\HT^{p}(\Rn)$ is independent of the choice of low-frequency cutoff $q$. One has 
\[
\HT^{s,p}(\Rn)=W^{s,p}(\Rn)=\lb D\rb^{-s}L^{p}(\Rn)
\]
for $1<p<\infty$ and $s\in\R$, as well as $\HT^{s,p}(\Rn)^{*}=\HT^{-s,p'}(\Rn)$ for $1\leq p<\infty$, and $\HT^{\infty}(\Rn)=\bmo(\Rn)$.


The following result, already mentioned in \eqref{eq:Sobolevintro} for $1<p<\infty$, extends \cite[Theorem 4.2]{Smith98a} and \cite[Theorem 7.4]{HaPoRo20}. 
Recall that 
\[
s(p):=\frac{n-1}{2}\Big|\frac{1}{p}-\frac{1}{2}\Big|
\]
 for $0<p\leq \infty$.

\begin{theorem}\label{thm:Sobolev}
Let $0<p\leq \infty$ and $s\in\R$. Then
\begin{equation}\label{eq:Sobolev}
\HT^{s+s(p),p}(\Rn)\subseteq\Hps\subseteq\HT^{s-s(p),p}(\Rn)
\end{equation}
continuously.
\end{theorem}
\begin{proof}
We may suppose that $s=0$. Then the statement for $p\geq1$ is contained in \cite[Theorem 7.4]{HaPoRo20}. In fact, in \cite{HaPoRo20} the required statement is proved for $p=1$, after which one can appeal to interpolation and duality, by Propositions \ref{prop:HpFIOint} and \ref{prop:HpFIOdual} and because $\HT^{2}_{FIO}(\Rn)=L^{2}(\Rn)$. Hence we may consider $p<1$ in the remainder, for which we can follow the proof of the case where $p=1$, in \cite[Theorem 4.2]{Smith98a} and \cite[Theorem 7.4] {HaPoRo20}, with a few modifications.

For example, throughout, applications of H\"{o}lder's inequality in the proof for $p=1$ should be replaced by Jensen's inequality when $p<1$. Also, when dealing with the low-frequency components for $p<1$, one cannot simply rely on Young's inequality anymore, as was done for $p=1$. Instead, one can use that the subspace of $L^{p}(\Rn)$ consisting of functions with Fourier support in a fixed compact set forms an algebra under convolution, cf.~\cite[Proposition 1.5.3]{Triebel10}.

For $p=1$, the proof in \cite{Smith98a} and \cite{HaPoRo20} of the first embedding in \eqref{eq:Sobolev} uses the classical fractional integration theorem (see e.g.~\cite[p.~5818]{HaPoRo20}). To avoid subtleties regarding the use of Sobolev embeddings for $p<1$, one can instead use that, for each $s\in\R$, one has $\|\lb D\rb^{s}\Psi(\sigma D)g\|_{L^{2}(\Rn)}\eqsim \sigma^{-s}\|\Psi(\sigma D)g\|_{L^{2}(\Rn)}$ for all $\sigma\in(0,1)$ and $g\in L^{2}(\Rn)$, where $\Psi$ is as in \eqref{eq:Psi}. With these modifications, to prove the first embedding in \eqref{eq:Sobolev} for $p<1$, one can proceed just as in \cite{HaPoRo20}.

For the second embedding, first note that 
\begin{align*}
\|\rho(D)^{2}f\|_{\HT^{-s(p),p}(\Rn)}&=\|\lb D\rb^{-s(p)}q(D)\rho(D)^{2}f\|_{L^{p}(\Rn)}\lesssim \|\rho(D)f\|_{L^{p}(\Rn)}\\
&\lesssim \|W_{l}f\|_{T^{p}(\Sp)}\leq \|Wf\|_{T^{p}(\Sp)}=\|f\|_{\Hp}
\end{align*}
for every $f\in\Hp$, by Lemma \ref{lem:Wlow}. Hence, by Theorem \ref{thm:moldecomp}, it suffices to show that the collection of coherent $\Hp$ molecules of type $N$, for sufficiently large $N\geq0$, associated with balls of radius no larger than a fixed number, is bounded in $\HT^{-s(p),p}(\Rn)$. For the low-frequency term that arises in this manner, one can argue as in the proof of Theorem \ref{thm:moldecomp} \eqref{it:moldecomp1}. On the other hand, for the high-frequency component one can follow the proof of \cite[Lemma 4.3]{Smith98a}. In fact, a slightly simpler version of that argument suffices, using the assumptions on the Fourier support of a coherent molecule.
\end{proof}

\begin{remark}\label{rem:Sobolevdual}
By duality, it follows from Theorem \ref{thm:Sobolev} and \cite[Theorems 2.5.8.1 and 2.11.3]{Triebel10} that the following embeddings involving Besov spaces hold:
\[
B^{n(\frac{1}{p}-1)-s+s(p)}_{\infty,\infty}(\Rn)\subseteq (\Hps)^{*}\subseteq B^{n(\frac{1}{p}-1)-s-s(p)}_{\infty,\infty}(\Rn),
\]
for all $0<p<1$ and $s\in\R$.
\end{remark}

We can use the same reasoning as in \cite[Corollary 7.6]{HaPoRo20} to obtain from Theorem \ref{thm:Sobolev} an equivalent quasi-norm on $\Hps$. Recall from \eqref{eq:As} and \eqref{eq:defWh} that
\[
\A_{s}(W_{h}f)(x,\w)=\Big(\int_{0}^{1}\fint_{B_{\sqrt{\sigma}}(x,\w)}|\psi_{\nu,\sigma}(D)f(y)|^{2}\ud y\ud\nu\frac{\ud\sigma}{\sigma^{1+2s}}\Big)^{1/2}
\]
for all $s\in\R$, $f\in\Sw'(\Rn)$ and $(x,\w)\in\Sp$, and similarly using the functional $\Ca_{s}$ from \eqref{eq:Cs}. Also recall that $q\in C^{\infty}_{c}(\Rn)$ satisfies $q(\xi)=1$ for $|\xi|\leq 2$.

\begin{corollary}\label{cor:equivalentnorm}
Let $0<p\leq \infty$ and $s\in\R$. Then there exists a $C>0$ such that the following assertions hold for each $f\in\Sw'(\Rn)$.
\begin{enumerate}
\item If $p<\infty$, then $f\in\Hps$ if and only if $q(D)f\in L^{p}(\Rn)$ and $\A_{s}(W_{h}f)\in L^{p}(\Sp)$, in which case 
\[
\frac{1}{C}\|f\|_{\Hps}\leq \|q(D)f\|_{L^{p}(\Rn)}+\|\A_{s}(W_{h}f)\|_{L^{p}(\Sp)}\leq C\|f\|_{\Hps}.
\]
\item If $p=\infty$, then $f\in\HT^{s,\infty}_{FIO}(\Rn)$ if and only if $q(D)f\in L^{\infty}(\Rn)$ and $\Ca_{s}(W_{h}f)\in L^{\infty}(\Sp)$, in which case
\[
\frac{1}{C}\|f\|_{\HT^{s,\infty}_{FIO}(\Rn)}\leq \|q(D)f\|_{L^{\infty}(\Rn)}+\|\Ca_{s}(W_{h}f)\|_{L^{\infty}(\Sp)}\leq C\|f\|_{\HT^{s,\infty}_{FIO}(\Rn)}.
\]
\end{enumerate}
\end{corollary}

\begin{remark}\label{rem:equivnorm}
By reasoning as in \cite[Section 3.2]{Rozendaal21}, one can characterize $\Hp$ using a square function that captures the high frequencies of an $f\in\Sw'(\Rn)$ in terms of the wave packets $\theta_{\w,\sigma}$ from \eqref{eq:theta}:
\begin{equation}\label{eq:defS}
S_{h}f(x,\w):=\Big(\int_{0}^{1}\fint_{B_{\sqrt{\sigma}}(x,\w)}|\theta_{\nu,\sigma}(D)f(y)|^{2}\ud y\ud\nu\frac{\ud\sigma}{\sigma}\Big)^{1/2}
\end{equation}
for $(x,\w)\in\Sp$. Namely, for $0<p<\infty$, one has $f\in\Hp$ if and only if $q(D)f\in L^{p}(\Rn)$ and $S_{h}f\in L^{p}(\Sp)$, in which case
\begin{equation}\label{eq:equivnormtheta}
\|f\|_{\Hp}\eqsim \|q(D)f\|_{L^{p}(\Rn)}+\|S_{h}f\|_{L^{p}(\Sp)}.
\end{equation}
The natural modification holds for $p=\infty$ (see \cite[Corollary 3.8]{Rozendaal21} for $p\geq1$).

To prove \eqref{eq:equivnormtheta}, one first uses Corollary \ref{cor:FIOtent} to obtain a variation 
of Proposition \ref{prop:HpFIOtent} in terms of the $\theta_{\w,\sigma}$ (see \cite[Proposition 3.6]{Rozendaal21}). Then Theorem \ref{thm:Sobolev} can be applied in the same way as above to derive \eqref{eq:equivnormtheta}. 
\end{remark}

Finally, we include, as a direct corollary of Propositions \ref{prop:tentembedding} and \ref{prop:HpFIOdual}, a proposition about embeddings between Hardy spaces for Fourier integral operators with different integrability parameters.

\begin{proposition}\label{prop:fracintHpFIO}
Let $0<p_{0}\leq p_{1}\leq \infty$, $0<p_{2}<1$ and $s\in\R$. Then 
\[
\HT^{s+n(\frac{1}{p_{0}}-\frac{1}{p_{1}}),p_{0}}_{FIO}(\Rn)\subseteq \HT^{s,p_{1}}_{FIO}(\Rn)
\]
and 
\[
\HT^{-s+n(\frac{1}{p_{0}}+\frac{1}{p_{2}}-1),p_{0}}_{FIO}(\Rn)\subseteq (\HT^{s,p_{2}}_{FIO}(\Rn))^{*},
\]
continuously.
\end{proposition}

\subsection{Invariance under Fourier integral operators}\label{subsec:invariance`}

The following theorem extends \cite[Theorem 6.10]{HaPoRo20} and \cite[Proposition 3.3]{Rozendaal22} to $p<1$.

\begin{theorem}\label{thm:FIObdd}
Let $0<p\leq\infty$ and $s\in\R$, and let $T$ be a Fourier integral operator of order $0$ and type (1/2,1/2,1) in standard form, associated with a global canonical graph, with symbol $a\in S^{0}_{1/2,1/2,1}$ and phase function $\Phi\in C^{\infty}(\R^{2n}\setminus o)$. Suppose that one of the following conditions holds:
\begin{enumerate}
\item\label{it:FIObdd1} $p>\frac{n}{n+1}$;
\item\label{it:FIObdd2} there exists a compact $K\subseteq\Rn$ such that $a(x,\eta)=0$ for all $(x,\eta)\in\R^{2n}$ with $x\notin K$;
\item\label{it:FIObdd3} $(x,\eta)\mapsto \Phi(x,\eta)$ is linear in $\eta$.
\end{enumerate}
Then there exists a $C\geq0$ such that $\|Tf\|_{\Hps}\leq C\|f\|_{\Hps}$ for all $f\in\Sw(\Rn)$.
\end{theorem}
\begin{proof}
By Corollary \ref{cor:FIOtent}, $WT(1-q(D))V\in\La(T^{p}_{s}(\Sp))$. Hence Proposition \ref{prop:HpFIOtent} and \eqref{eq:repro} imply that $T(1-q(D))\in \La(\Hps)$. The same reasoning shows that $Tq(D)\in \La(\Hps)$ in case \eqref{it:FIObdd3}.

On the other hand, in cases \eqref{it:FIObdd1} and \eqref{it:FIObdd2}, $\|Tq(D)f\|_{\HT^{s+s(p),p}(\Rn)}\lesssim \|f\|_{\HT^{s-s(p),p}(\Rn)}$ for all $f\in\Sw(\Rn)$, by \cite[Theorem 6.5]{IsRoSt21}. Hence Theorem \ref{thm:Sobolev} yields
\[
\|Tq(D)f\|_{\Hps}\lesssim \|Tq(D)f\|_{\HT^{s+s(p),p}(\Rn)}\lesssim \|f\|_{\HT^{s-s(p),p}(\Rn)}\lesssim \|f\|_{\Hps}.\qedhere
\]
\end{proof}

\begin{remark}\label{rem:extensioninfty}
By density, $T$ extends uniquely to a bounded operator on $\Hps$ if $p<\infty$.  
Moreover, if $T$ is a pseudodifferential operator, then the adjoint of $T$ is continuous on $\Sw(\Rn)$, so that $T$ extends to an operator on $\Sw'(\Rn)$ by adjoint action. Then \eqref{it:FIObdd3} shows that $T\in\La(\HT^{s,\infty}_{FIO}(\Rn))$. 
\end{remark}

\begin{remark}\label{rem:FIOlocalbdd}
One can also extend \cite[Theorem 3.7]{Smith98a} and \cite[Proposition 2.3]{LiRoSoYa24} to $p<1$, by showing that a compactly supported Fourier integral operator $T$ of order zero as in \cite[Section 25.2]{Hormander09}, associated with a local canonical graph, is bounded on $\Hps$ for all $0<p\leq\infty$ and $s\in\R$. To do so, recall that $T$ is a finite sum of a product of operators which can be represented as in \eqref{eq:oscint} and have symbols with compact spatial support, plus an operator with a Schwartz kernel which is a Schwartz function (see \cite[Proposition 2.14]{HaPoRo20}). An operator of the latter kind is bounded on $\Hps$ by Theorem \ref{thm:Sobolev}. On the other hand, an operator as in \eqref{eq:oscint} whose symbol has compact spatial support can be dealt with in a similar way as above, using a modification of Corollary \ref{cor:FIOtent} (see \cite[Corollary 5.4]{HaPoRo20}) to deal with the high-frequency component, and \cite[Theorem 6.5]{IsRoSt21} for the low frequencies. 
\end{remark}

By combining Theorems \ref{thm:Sobolev} and \ref{thm:FIObdd}, for Hardy spaces we obtain the following extension of \cite[Theorem 6.1]{IsRoSt21}, to a larger class of symbols but under the additional assumption in \eqref{it:phase4} of Definition \ref{def:operator} (which is automatically satisfied for $n\geq3$).

\begin{corollary}\label{cor:FIObdd}
Let $0<p\leq\infty$ and $s\in\R$, and let $T$ be as in Theorem \ref{thm:FIObdd}. Then there exists a $C\geq0$ such that $\|Tf\|_{\HT^{s-s(p),p}(\Rn)}\leq C\|f\|_{\HT^{s+s(p),p}(\Rn)}$ for all $f\in\Sw(\Rn)$.
\end{corollary}


\begin{remark}\label{rem:Euclidwave}
The condition on $p$ in \eqref{it:FIObdd1} of Theorem \ref{thm:FIObdd} is sharp in general, given that it is necessary in Corollary \ref{cor:FIObdd} for $T=e^{i\sqrt{-\Delta}}$, even in dimension $n=1$ (see \cite[Section 9]{IsRoSt21}). In particular, the Cauchy problem associated with the Euclidean half-wave equation $(\partial_{t}-i\sqrt{-\Delta})u(t)=0$ is not well posed on $\Hps$ for any $0<p\leq n/(n+1)$ and $s\in\R$. 
The obstacle is insufficient integrability of the inverse Fourier transform of $\eta\mapsto e^{i|\eta|}q(\eta)$, resulting from the lack of smoothness at $\eta=0$.

On the other hand, $\eta\mapsto \cos(t|\eta|)q(\eta)$ and $\eta\mapsto |\eta|^{-1}\sin(t|\eta|)q(\eta)$ are smooth at zero for all $t\in\R$, and their inverse Fourier transforms lie in $L^{p}(\Rn)$ for all $p>0$. 
Hence, by the proof of Theorem \ref{thm:FIObdd} and \cite[Proposition 1.5.3]{Triebel10}, $\cos(t\sqrt{-\Delta})\in\La(\Hps)$ and $\frac{\sin(t\sqrt{-\Delta})}{\sqrt{-\Delta}}\in\La(\HT^{s-1,p}_{FIO}(\Rn),\Hps)$ for all 
$s\in\R$, and the Cauchy problem for the Euclidean wave equation is well posed on $\Hps$. By combining this with Theorem \ref{thm:Sobolev} and using that the low-frequency components of Triebel--Lizorkin and Besov space quasi-norms are equivalent, one sees that the main result of \cite{IsRoSt21} for the Euclidean wave equation in fact extends to all $p>0$. 
\end{remark}

\subsection{Equivalent characterizations}\label{subsec:charac}

In this subsection we prove an extension of the main results of \cite{FaLiRoSo23,Rozendaal21} to $p<1$, by giving an equivalent characterization of $\Hp$ in terms of the parabolic localizations $\ph_{\w}$ from \eqref{eq:phw}.

For $N,\sigma>0$, we will work with the maximal function $M_{N,\sigma}^{*}$, given by
\[
M_{N,\sigma}^{*}{g}(x,\omega)
:=\sup_{(y,\nu)\in \Sp}(1+\sigma^{-1}d((x,\omega),(y,\nu))^{2})^{-N}|{g}(y,\nu)|
\]
for ${g}:\Sp\to\C$ and $(x,\w)\in\Sp$. Also write
\[
W_{\sigma}f(x,\w):=\theta_{\w,\sigma}(D)f(x)
\]
for $f\in\Sw'(\Rn)$, where $\theta_{\w,\sigma}$ is as in \eqref{eq:theta}. The following proposition, relating $M_{N,\sigma}^{*}$ to the maximal function $\Ma_{\la}$ from \eqref{eq:maxHL}, is crucial for the proof of our equivalent characterization, especially for $p\leq 1$.

\begin{proposition}\label{prop:maxineq}
Let $\lambda>0$ and $N>n/\la$. Then there exists a $C\geq0$ such that
\[
M_{N,\sigma}^{*}(W_{\sigma}f)(x,\omega)\leq C \mathcal{M}_{\lambda}(W_{\sigma}f)(x,\omega)
\]
for all $f\in \Sw'(\Rn)$, $(x,\w)\in\Sp$ and $\sigma\in(0,1)$.
\end{proposition}
\begin{proof}
Fix $f\in\Sw'(\Rn)$, $(x,\w)\in\Sp$ and $\sigma\in (0,1)$. 
%
First note that, by \eqref{eq:phiproperties3} and Remark \ref{rem:thetatilde}, 
\begin{equation}\label{eq:reprotheta}
\theta_{\nu,\sigma}(D)f(y)=\sigma^{-\frac{n-1}{4}}\int_{S^{n-1}}\tilde{\theta}_{\nu,\sigma}(D)\theta_{\mu,\sigma}(D)f(y) \ud\mu
\end{equation}
for all $(y,\nu)\in\Sp$. In fact, by Lemma \ref{lem:packetbounds} and Remark \ref{rem:thetatilde}, the integrand is only nonzero where $|\mu-\nu|\leq 4\sqrt{\sigma}$. In the latter case, Remark \ref{rem:thetatilde} and \eqref{eq:equivmetric} yield
\begin{align*}
|\F^{-1}(\tilde{\theta}_{\nu,\sigma})(y-z)|&\lesssim \sigma^{-\frac{3n+1}{4}}(1+\sigma^{-1}d((y,\nu),(z,\nu))^{2})^{-N}\\
&\lesssim \sigma^{-\frac{3n+1}{4}}(1+\sigma^{-1}d((y,\nu),(z,\mu))^{2})^{-N}
\end{align*}
for all $z\in\Rn$. Hence we can use the triangle inequality to write
\begin{align*}
&(1+\sigma^{-1}d((x,\omega),(y,\nu))^{2})^{-N}|\theta_{\nu,\sigma}(D)f(y)|\\
&\lesssim\sigma^{-n}\int_{\Sp}\frac{|\theta_{\mu,{\sigma}}(D)f(z)|}{(1+\sigma^{-1}d((x,\omega),(y,\nu))^{2})^{N}(1+\sigma^{-1}d((y,\nu),(z,\mu))^{2})^{N}}
\ud z\ud\mu\\
&\lesssim \sigma^{-n}\int_{\Sp}\frac{|
\theta_{\mu,\sigma}(D)f(z)|^{1-\lambda}|\theta_{\mu,\sigma}(D)f(z)|^{\lambda} }{(1+\sigma^{-1}d((x,\omega),(z,\mu))^{2})^{N}}\ud z\ud\mu.
\end{align*}
By combining this with \eqref{eq:HLcontrol}, we see that
\begin{align*}
&M_{N,\sigma}^{*}(W_{\sigma}f)(x,\omega)\\
&\lesssim \sigma^{-n}\int_{\Sp}\frac{|
\theta_{\mu,\sigma}(D)f(z)|^{1-\lambda} |\theta_{\mu,\sigma}(D)f(z)|^{\lambda}}{(1+\sigma^{-1}d((x,\omega),(z,\mu))^{2})^{N(1-\la)}(1+\sigma^{-1}d((x,\omega),(z,\mu))^{2})^{N\la}}\ud z\ud\mu\\
&\leq \big(M_{N,\sigma}^{*}(W_{\sigma}f)(x,\omega)\big)^{1-\lambda}
{\sigma}^{-n}\int_{\Sp}\frac{|\theta_{\mu,\sigma}(D)f(z)|^{\lambda}}{(1+\sigma^{-1}d((x,\omega),(z,\mu))^{2})^{N\lambda}}\ud z\ud\mu\\
&\leq \big(M_{N,\sigma}^{*}(W_{\sigma}f)(x,\omega)\big)^{1-\lambda}
\big(\Ma_{\lambda}(W_{\sigma}f)(x,\omega)\big)^{\lambda}.
\end{align*}
This proves the required statement if $M_{N,\sigma}^{*}(W_{\sigma}f)(x,\omega)<\infty$, and it remains to show that $M_{N,\sigma}^{*}(W_{\sigma}f)(x,\w)$ is finite whenever $\mathcal{M}_{\lambda}(W_{\sigma}f)(x,\omega)$ is finite.

To this end, first write 
\[
\theta_{\nu,\sigma}(D)f(y)=\theta_{\nu,\sigma}(D)\wt{\Psi}(\sigma D)f(y)
=\int_{\Rn}\F^{-1}\theta_{\nu,\sigma}(y-z)\wt{\Psi}(\sigma D)f(z)\ud z
\]
for $(y,\nu)\in\Sp$, where we recall from Remark \ref{rem:thetatilde} that $\wt{\Psi}\equiv1$ on $\supp(\Psi)$. Since $\wt{\Psi}\in \Sw(\Rn)$ and $f\in\Sw'(\Rn)$, there exists an $M\geq0$, dependent on $f$, such that
\[
\sup_{z\in\Rn}\lb z\rb^{-2M}|\wt{\Psi}(\sigma D)f(z)|<\infty.
\]
Hence Lemma \ref{lem:packetbounds} implies that
\begin{equation}\label{eq:decaydependent}
\begin{aligned}
&\lb y\rb^{-2M}|\theta_{\nu,\sigma}(D)f(y)|\\
&\lesssim\int_{\Rn} \lb y-z\rb^{2M}|\F^{-1}\theta_{\nu,\sigma}(y-z)|\ud z\sup_{z\in\Rn} \lb z\rb^{-2M}|\wt{\Psi}(\sigma D)f(z)|\\
 &\lesssim \sup_{z\in\Rn} \lb z\rb^{-2M}|\wt{\Psi}(\sigma D)f(z)|\lesssim 1,
\end{aligned}
\end{equation}
where the implicit constants depend on $f$ and $\sigma$, but not on $(y,\nu)$.

Next, for $0<\veps<1$, denote 
\[
\sup_{(y,\nu)\in\Sp}\frac{\big|\theta_{\nu,{\sigma}}(D)f(y)\big|}{(1+\sigma^{-1}d((x,\omega)(y,\nu))^{2})^{N}(1+\veps\sigma^{-1} d((x,\omega),(y,\nu))^{2})^{M}}
\]
by $M_{N,\sigma,\veps}^{*}(W_{\sigma}f)(x,\omega)$.
Then, by \eqref{eq:equivmetric} and \eqref{eq:decaydependent},
\begin{equation}\label{eq:Mepsfinite}
\begin{aligned}
&M_{N,\sigma,\veps}^{*}(W_{\sigma}f)(x,\omega)\\
&\leq\sup_{(y,\nu)\in\Sp}(1+\veps\sigma^{-1} d((x,\omega),(y,\nu))^{2})^{-M}|\theta_{\nu,\sigma}(D)f(y)|\\
&\lesssim \sup_{(y,\nu)\in\Sp}\lb x-y\rb^{-2M}|\theta_{\nu,\sigma}(D)f(y)|\\
&\lesssim\lb x\rb^{2M}\sup_{(y,\nu)\in\Sp}\lb y\rb^{-2M}|\theta_{\nu,{\sigma}}(D)f(y)|<\infty,
\end{aligned}
\end{equation}
for implicit constants dependent on $f$, $\sigma$ and $\veps$. 

We can now combine \eqref{eq:reprotheta} and Lemma \ref{lem:packetbounds} in the same way as before to write
\begin{align*}
&|\theta_{\nu,\sigma}(D)f(y)|\lesssim\sigma^{-n}
\int_{\Sp}(1+\sigma^{-1}d((y,\nu),(z,\mu))^{2})^{-N-M}|\theta_{\mu,{\sigma}}(D)f(z)|\ud z\ud\mu\\
&\leq \sigma^{-n}
\int_{\Sp}\frac{|\theta_{\mu,\sigma}(D)f(z)|}{(1+\sigma^{-1}d((y,\nu),(z,\mu))^{2})^{N}(1+\veps\sigma^{-1}d((y,\nu),(z,\mu))^{2})^{M}}\ud z\ud\mu,
\end{align*}
for an implicit constant independent of $(y,\nu)$ and $\veps$. The triangle inequality then yields
\begin{align*}
&(1+\sigma^{-1}d((x,\omega),(y,\nu))^{2})^{-N}(1+\veps\sigma^{-1}d((x,\omega),(y,\nu))^{2})^{-M}|\theta_{\nu,\sigma}(D)f(y)|\\
&\lesssim\sigma^{-n}\int_{\Sp}\frac{|\theta_{\mu,\sigma}(D)f(z)|}{(1+\sigma^{-1}d((x,\w),(z,\mu))^{2})^{N}(1+\veps\sigma^{-1}d((x,\w),(z,\mu))^{2})^{M}}\ud z\ud\mu\\
&\leq \big(M_{N,\sigma,\veps}^{*}(W_{\sigma}f)(x,\omega)\big)^{1-\lambda}
\sigma^{-n}\int_{\Sp}\frac{|\theta_{\mu,{\sigma}}(D)f(z)|^{\lambda}}{(1+\sigma^{-1}d((x,\omega),(z,\mu))^{2})^{N\lambda}}
 \ud z\ud\mu.
\end{align*}
This in turn implies that
\begin{align*}
M_{N,\sigma,\veps}^{*}(W_{\sigma}f)(x,\omega)\lesssim \big(M_{N,\sigma,\veps}^{*}(\W_{\sigma}f)(x,\omega)\big)^{1-\lambda}\big(\Ma_{\lambda}(W_{\sigma}f)(x,\omega)\big)^{\lambda},
\end{align*}
and combined with \eqref{eq:Mepsfinite} we now obtain
\[
M_{N,\sigma,\veps}^{*}(W_{\sigma}f)(x,\omega)\lesssim\Ma_{\lambda}(W_{\sigma}f)(x,\omega).
\]
Finally, since the implicit constant is independent of $\veps$, letting $\varepsilon$ tend to zero yields
\begin{align*}
M_{N,\sigma}^{*}(W_{\sigma}f)(x,\omega)\lesssim\Ma_{\lambda}(W_{\sigma}f)(x,\omega).
\end{align*}
This shows that $M_{N,\sigma}^{*}(W_{\sigma}f)(x,\omega)$ is indeed finite if $\Ma_{\lambda}(W_{\sigma}f)(x,\omega)$ is finite. 
\end{proof}


\begin{corollary}\label{cor:maxineq}
Let $\lambda>0$ and $N>n/\la$. Then there exists a $C\geq0$ such that
\[
\sigma^{-n}\int_{\Sp}\frac{|\theta_{\nu,{\sigma}}(D)f(y)|}{(1+{\sigma}^{-1}d((x,\omega),(y,\nu))^{2})^{N}}\ud y\ud\nu
\leq C\Ma_{\lambda}(W_{\sigma}f)(x,\omega)
\]
for all $f\in \Sw'(\Rn)$, $(x,\w)\in\Sp$ and $\sigma\in(0,1)$.
\end{corollary}
\begin{proof}
By Proposition \ref{prop:maxineq} and \eqref{eq:HLcontrol}, 
\begin{align*}
&\sigma^{-n}\int_{\Sp}\frac{|\theta_{\nu,{\sigma}}(D)f(y)|}{(1+\sigma^{-1}d((x,\omega),(y,\nu))^{2})^{N}}\ud y\ud\nu\\
&\lesssim \big(M_{N,\sigma}^{*}(W_{\sigma}f)(x,\omega)\big)^{1-\lambda}\sigma^{-n}\int_{\Sp}\frac{|\theta_{\nu,{\sigma}}(D)f(y)|^{\lambda}}{(1+\sigma^{-1}d((x,\omega),(y,\nu))^{2})^{N\lambda}}\ud y\ud\nu\\
&\leq \big(M_{N,\sigma}^{*}(W_{\sigma}f)(x,\omega)\big)^{1-\lambda}\big(\Ma_{\lambda}(W_{\sigma}f)(x,\omega)\big)^{\lambda}\leq \Ma_{\lambda}(W_{\sigma}f)(x,\omega).\qedhere
\end{align*}
\end{proof}


We can now prove the main result of this subsection.

\begin{theorem}\label{thm:equivchar}
Let $0<p<\infty$ and $s\in\R$. Then there exists a constant $C>0$ such that the following holds. An $f\in\Sw'(\Rn)$ satisfies $f\in\Hps$ if and only if $q(D)f\in L^{p}(\Rn)$, $\ph_{\w}(D)f\in \HT^{s,p}(\Rn)$ for almost all $\w\in S^{n-1}$, and $(\int_{S^{n-1}}\|\ph_{\w}(D)f\|_{\HT^{s,p}(\Rn)}^{p}\ud \w)^{1/p}<\infty$, in which case
\[
\frac{1}{C}\|f\|_{\Hps}\!\leq\!\|q(D)f\|_{L^{p}(\Rn)}+\Big(\int_{S^{n-1}}\!\|\ph_{\w}(D)f\|_{\HT^{s,p}(\Rn)}^{p}\ud\w\!\Big)^{1/p}\!\leq\!C\|f\|_{\Hps}.
\]
\end{theorem}
\begin{proof}
The statement is contained in \cite[Theorem 1.1]{Rozendaal21} for $1<p<\infty$, and in \cite[Proposition 3.1]{FaLiRoSo23} for $p=1$. Hence it suffices to consider $0<p<1$. However, we will prove the required statement in a unified manner for all $0<p\leq 2$, using an argument similar to the one from \cite{FaLiRoSo23} for $p=1$. We may suppose that $s=0$.

First suppose that $f\in\Hp$. By Theorem \ref{thm:Sobolev} or Corollary \ref{cor:equivalentnorm}, $q(D)f\in L^{p}(\Rn)$ and $\|q(D)f\|_{L^{p}(\Rn)}\lesssim \|f\|_{\Hp}$. Moreover, recall that $\ph_{\w}(\xi)=0$ for $|\xi|\leq 1/2$, and that the $\HT^{p}(\Rn)$ quasi-norm is independent of the choice of low-frequency cutoff. Also note that $\ph_{\w}{q}\in\Sw(\Rn)$ for each $\w\in S^{n-1}$, with each of the Schwartz seminorms bounded uniformly in $\w$. Hence Theorem \ref{thm:Sobolev} and \ref{thm:FIObdd} yield
\begin{align*}
\|\ph_{\w}(D)q(D)f\|_{\HT^{p}(\Rn)}&=\|\lb D\rb^{s(p)}\ph_{\w}(D)q(D)f\|_{\HT^{-s(p),p}(\Rn)}\\
&\lesssim \|\lb D\rb^{s(p)}\ph_{\w}(D)q(D)f\|_{\Hp}\lesssim \|f\|_{\Hp},
\end{align*}
for an implicit constant independent of $\w\in S^{n-1}$. This in turn implies that $(\int_{S^{n-1}}\|\ph_{\w}(D)q(D)f\|_{\HT^{p}(\Rn)}^{p}\ud\w)^{1/p}\lesssim \|f\|_{\Hp}$. 

Finally, set $g:=(1-q(D))f$. Then \eqref{eq:Hp} and the properties of $q$ imply that
\begin{align*}
\|\ph_{\w}(D)(1-q(D))f\|_{\HT^{p}(\Rn)}&\eqsim\|\ph_{\w}(D)(1-q(D))f\|_{H^{p}(\Rn)}\\
&=\Big(\int_{\Rn}\Big(\int_{0}^{1}|\Psi(\sigma D)\ph_{\w}(D)g(x)|^{2}\frac{\ud\sigma}{\sigma}\Big)^{p/2}\ud x\Big)^{1/p}\\
&=\Big(\int_{\Rn}\Big(\int_{0}^{1}|\theta_{\w,\sigma}(D)g(x)|^{2}\frac{\ud\sigma}{\sigma}\Big)^{p/2}\ud x\Big)^{1/p}
\end{align*}
for each $\w\in S^{n-1}$. Hence \eqref{eq:vertical}, Remark \ref{rem:equivnorm} and Theorem \ref{thm:FIObdd} yield 
\begin{align*}
&\Big(\int_{S^{n-1}}\|\ph_{\w}(D)(1-q(D))f\|_{H^{p}(\Rn)}^{p}\ud\w\Big)^{1/p}\\
&=\Big(\int_{\Sp}\Big(\int_{0}^{1}|\theta_{\w,\sigma}(D)g(x)|^{2}\frac{\ud\sigma}{\sigma}\Big)^{p/2}\ud x\Big)^{1/p}\ud x\ud\w\Big)^{1/p}\\
&\lesssim \Big(\int_{\Sp}\Big(\int_{0}^{1}\fint_{B_{\sqrt{\sigma}}(x,\w)}|\theta_{\nu,\sigma}(D)g(y)|^{2}\ud y\ud\nu\frac{\ud\sigma}{\sigma}\Big)^{p/2}\ud x\ud\w\Big)^{1/p}\\
&\lesssim \|g\|_{\Hp}\lesssim \|f\|_{\Hp},
\end{align*}
which proves one of the required implications.

For the converse implication, suppose that $q(D)f\in L^{p}(\Rn)$, that $\ph_{\w}(D)f\in \HT^{p}(\Rn)$ for almost all $\w\in S^{n-1}$, and that $(\int_{S^{n-1}}\|\ph_{\w}(D)f\|_{\HT^{p}(\Rn)}^{p}\ud \w)^{1/p}<\infty$. Using notation as in \eqref{eq:defS}, by Remark \ref{rem:equivnorm} it suffices to show that $\|S_{h}f\|_{L^{p}(\Sp)}\lesssim (\int_{S^{n-1}}\|\ph_{\w}(D)f\|_{\HT^{p}(\Rn)}^{p}\ud \w)^{1/p}$. This part of the argument in fact works for all $p\in(0,\infty)$.

Let $N>n/p$ and $\la\in(0,\min(2,p))$ be such that $\lambda N>n$. By Proposition \ref{prop:maxineq},
\begin{align*}
S_{h}f(x,\omega)^{2}&\leq\int_{0}^{1}\sup_{(y,\nu)\in B_{\sqrt{\sigma}}(x,\omega)} |\theta_{\nu,{\sigma}}(D)f(y)|^{2}\frac{\ud\sigma}{\sigma}\lesssim \int_{0}^{1}\big(M_{N,\sigma}^{*}(W_{\sigma}f)(x,\omega)\big)^{2}\frac{\ud\sigma}{\sigma}\\
&\lesssim\int_{0}^{1}\big(\mathcal{M}_{\lambda}(W_{\sigma}f)(x,\omega)\big)^{2}\frac{\ud\sigma}{\sigma}
\end{align*}
for all $(x,\w)\in\Sp$. Hence we can use that the Hardy--Littlewood maximal function $\Ma$ is bounded on $L^{p/\la}(\Sp;L^{2/\la}(0,\infty))$:
\begin{align*}
\|S_{h}f\|_{L^{p}(\Sp)}&\lesssim\Big(\int_{\Sp}\Big(\int_{0}^{1}\big(\Ma(|W_{\sigma}f|^{\la})(x,\omega)\big)^{2/\la}\frac{\ud\sigma}{\sigma}\Big)^{p/2}\ud x\ud\w\Big)^{1/p}\\
&\lesssim \Big(\int_{\Sp}\Big(\int_{0}^{1}|\theta_{\w,\sigma}(D)f(x)|^{2}\frac{\ud\sigma}{\sigma}\Big)^{p/2}\ud x\ud\w\Big)^{1/p}\\
&\leq\Big(\int_{S^{n-1}}\|\ph_{\w}(D)f\|_{H^{p}(\Rn)}^{p}\ud\w\Big)^{1/p}\\
&\eqsim \Big(\int_{S^{n-1}}\|\ph_{\w}(D)f\|_{\HT^{p}(\Rn)}^{p}\ud\w\Big)^{1/p},
\end{align*}
thereby concluding the proof.
\end{proof}

\begin{remark}\label{rem:anisHardy}
Since the definition of $\HT^{p}(\Rn)$ is independent of the choice of low-frequency cutoff, it is straightforward to see that, for all $0<p<\infty$ and $s\in\R$, 
\[
\|f\|_{\Hps}\eqsim \|q(D)f\|_{L^{p}(\Rn)}+\Big(\int_{S^{n-1}}\|\lb D\rb^{s}\ph_{\w}(D)f\|_{H^{p}(\Rn)}^{p}\ud\w\Big)^{1/p}.
\]
Less trivially, as in \cite{HaPoRoYu26}, one can show that 
\[
\|f\|_{\Hp}\eqsim \|q(D)f\|_{L^{p}(\Rn)}+\Big(\int_{S^{n-1}}\|\ph_{\w}(D)f\|_{H^{p}_{\w}(\Rn)}^{p}\ud\w\Big)^{1/p}.
\]
Here $H^{p}_{\w}(\Rn)$ is a parabolic Hardy space in the direction of $\w\in S^{n-1}$ (see \cite{HaPoRoYu26,Calderon-Torchinsky75,Calderon-Torchinsky77}). This follows by proving, as in \cite{HaPoRoYu26}, that $\|\ph_{\w}(D)f\|_{H^{p}(\Rn)}\eqsim \|\ph_{\w}(D)f\|_{H^{p}_{\w}(\Rn)}$, due to the support properties of $\ph_{\w}$.
\end{remark}

\section*{Acknowledgments}\label{sec:acknowledge}

 The authors 
would like to thank Prof.~Lixin Yan for helpful  discussions and valuable support, and the referees for carefully reading the manuscript and for several useful suggestions.

\appendix

\section{Complex interpolation}\label{sec:inter}

Whereas the theory of real interpolation extends directly from Banach spaces to quasi-Banach spaces, the same does not hold for complex interpolation. The main problem here is that the maximum modulus principle can break down on quasi-Banach spaces. As a consequence, point evaluation does not necessarily constitute a continuous functional on the spaces of analytic functions typically used to define interpolation spaces. Various solutions to this problem have been proposed in the literature; we will rely on ideas from \cite{Kalton-Mitrea98}. 

Throughout, all quasi-norms under consideration are assumed to be continuous. {We note that, while a quasi-norm need not be continuous, every quasi-Banach space can be equipped with an equivalent quasi-norm which is a so-called \emph{$p$-norm}, and as such it is continuous (see \cite{Rolewicz84}).}

For $X$ a quasi-Banach space and $\Omega\subseteq \C$ an open subset, an $f:\Omega\to X$ is \emph{analytic} if, for every $z_{0}\in\Omega$, there exist $c>0$ and $(x_{k})_{k=0}^{\infty}\subseteq X$ such that
\[
f(z)=\sum_{k=0}^{\infty}(z-z_{0})^{k}x_{k}
\]
for all $z\in\Omega$ with $|z-z_{0}|<c$, where the series converges uniformly. This notion of analyticity is more useful than that of complex differentiability (see \cite{Kalton86a}).

Now let $X_{0}$ and $X_{1}$ be quasi-Banach spaces, continuously embedded into a Hausdorff topological vector space $Z$. We call $(X_{0},X_{1})$ an \emph{interpolation couple}. Set $\St:=\{z\in\C\mid 0<\Real(z)<1\}$, and let $\F_{X_{0},X_{1}}$ consist of all analytic $f:\St\to X_{0}+X_{1}$ that extend in a continuous and bounded manner to $\overline{\St}$, and that are continuous and bounded from $\{z\in \C\mid \Real(z)=j\}$ to $X_{j}$, for $j\in\{0,1\}$. Let
\begin{equation}\label{eq:analytnorm}
\|f\|_{\F_{X_{0},X_{1}}}:=\sup_{t\in\R}\|f(it)\|_{X_{0}}+\sup_{t\in\R}\|f(1+it)\|_{X_{1}}+\sup_{z\in\St}\|f(z)\|_{X_{0}+X_{1}}
\end{equation}
be the natural quasi-norm on $\F_{X_{0},X_{1}}$. 

Note that the last term in \eqref{eq:analytnorm} is typically not included in the theory of complex interpolation on Banach spaces; due to the maximum modulus principle, it is superfluous. However, on quasi-Banach spaces, this term ensures that $\F_{X_{0},X_{1}}$ is complete, i.e.~that $\F_{X_{0},X_{1}}$ is itself a quasi-Banach space, and that each point evaluation $e_{z}:\F_{X_{0},X_{1}}\to\C$, at $z\in\St$, is continuous. As such, for $\theta\in(0,1)$, we can set $[X_{0},X_{1}]_{\theta}:=\{f(\theta)\mid f\in \F_{X_{0},X_{1}}\}$, endowed with the quasi-norm
\[
\|x\|_{[X_{0},X_{1}]_{\theta}}:=\inf\{\|f\|_{\F_{X_{0},X_{1}}}\mid f\in \F_{X_{0},X_{1}},f(\theta)=x\}.
\]
Then (see also \cite[Lemma 2.1]{Kalton-Mitrea98})
\[
[X_{0},X_{1}]_{\theta}=\F_{X_{0},X_{1}}/\ker(e_{\theta})
\]
is a quasi-Banach space as well.

Now, although the addition of the last term in \eqref{eq:analytnorm} provides an easy way to produce well-defined interpolation spaces, for a more complete interpolation theory it is natural to work with the class of analytically convex quasi-Banach spaces, introduced in \cite{Kalton86b} and then applied to interpolation theory in \cite{Kalton-Mitrea98}.

For our purposes (see also \cite[Theorem 7.4]{KaMaMi07}), a quasi-Banach space $X$ is \emph{analytically convex} if there exists a $C\geq0$ such that 
\[
\sup\{\|f(z)\|_{X} \mid  z\in\St\}\leq C\sup\{\|f(z)\|_{X}\mid  \Real(z)\in \{0, 1\}\}
\]
for each analytic $f:\St\to X$ that extends continuously and boundedly to $\overline{\St}$. Every Banach space is analytically convex, with $C=1$, as are the quasi-Banach spaces that interest us in this article. 

For example, the notion of analytic convexity plays a role in the following lemma, an extension of Stein's interpolation result to quasi-Banach spaces that we record for independent interest.

\begin{lemma}\label{lem:Steininter}
Let $(X_{0},X_{1})$ and $(Y_{0},Y_{1})$ be interpolation couples, and let 
$(T_{z})_{z\in\overline{\St}}\subseteq\La(X_{0}+X_{1},Y_{0}+Y_{1})$ be such that the following conditions hold.
\begin{enumerate}
\item\label{it:Steininter1} $K:=\sup_{z\in\overline{\St}}\|T_{z}\|_{\La(X_{0}+X_{1},Y_{0}+Y_{1})}<\infty$;
\item\label{it:Steininter2} For all $g\in X_{0}+X_{1}$, the $Y_{0}+Y_{1}$-valued map $z\mapsto T_{z}g$ is continuous on $\overline{\St}$, and analytic on $\St$;
\item\label{it:Steininter3} For each $j\in\{0,1\}$, one has $T_{j+it}\in\La(X_{j},Y_{j})$ for all $t\in\R$, and $K_{j}:=\sup_{t\in\R}\|T_{j+it}\|_{\La(X_{j},Y_{j})}<\infty$.
\item\label{it:Steininter4} For each $j\in\{0,1\}$ and all $g\in X_{j}$, the $Y_{j}$-valued map $t\mapsto T_{j+it}g$ is continuous on $\R$.
\end{enumerate}
Then $T_{\theta}\in \La([X_{0},X_{1}]_{\theta},[Y_{0},Y_{1}]_{\theta})$ for all $\theta\in(0,1)$, and 
\begin{equation}\label{eq:Steininter1}
\|T_{\theta}\|_{\La([X_{0},X_{1}]_{\theta},[Y_{0},Y_{1}]_{\theta})}\\
\leq \max(K_{0},K_{1},K).
\end{equation}
Moreover, if $Y_{0}+Y_{1}$ is analytically convex, then there exists a $C\geq0$ such that
\begin{equation}\label{eq:Steininter2}
\|T_{\theta}\|_{\La([X_{0},X_{1}]_{\theta},[Y_{0},Y_{1}]_{\theta})}\leq C\max(K_{0},K_{1})
\end{equation}
for all $\theta\in(0,1)$.
\end{lemma}
\begin{proof}
First note that $[X_{0},X_{1}]_{\theta}\subseteq X_{0}+X_{1}$ for all $\theta\in(0,1)$. Hence, if we show that $[z\mapsto T_{z}f(z)]\in \F_{Y_{0},Y_{1}}$ for each $f\in\F_{X_{0},X_{1}}$, and that
\[
\|[z\mapsto T_{z}f(z)]\|_{\F_{Y_{0},Y_{1}}}\leq \max(K_{0},K_{1},K)\|f\|_{\F_{X_{0},X_{1}}},
\] 
then \cite[Lemma 2.2]{Kalton-Mitrea98} yields \eqref{eq:Steininter1}. Moreover, if we show that 
\[
\|[z\mapsto T_{z}f(z)]\|_{\F_{Y_{0},Y_{1}}}\lesssim \max(K_{0},K_{1})\|f\|_{\F_{X_{0},X_{1}}}
\]
whenever $Y_{0}+Y_{1}$ is analytically convex, then \cite[Lemma 2.2]{Kalton-Mitrea98} also yields \eqref{eq:Steininter2}.

Let $f\in\F_{X_{0},X_{1}}$. Then \eqref{it:Steininter1} implies that $T_{z}f(z)\in Y_{0}+Y_{1}$ for all $z\in\overline{\St}$, and that 
\begin{equation}\label{eq:Steininter3}
\sup_{z\in\overline{\St}}\|T_{z}f(z)\|_{Y_{0}+Y_{1}}\leq K\sup_{z\in\overline{\St}}\|f(z)\|_{X_{0}+X_{1}}.
\end{equation}
By combining \eqref{it:Steininter1} and the first part of \eqref{it:Steininter2}, one also sees that the $Y_{0}+Y_{1}$-valued map $z\mapsto T_{z}f(z)$ is continuous on $\overline{\St}$. And, for each $z_{0}\in\St$, the second part of \eqref{it:Steininter2} implies that there exist $(x_{k})_{k=0}^{\infty}\subseteq X_{0}+X_{1}$ and $(y_{k,l})_{k,l=0}^{\infty}\subseteq Y_{0}+Y_{1}$ such that 
\[
T_{z}f(z)=T_{z}\Big(\sum_{k=0}^{\infty}(z-z_{0})^{k}x_{k}\Big)=\sum_{k=0}^{\infty}(z-z_{0})^{k}\sum_{l=0}^{\infty}(z-z_{0})^{l}y_{k,l}=\sum_{k,l=0}^{\infty}(z-z_{0})^{k+l}y_{k,l}
\]
for all $z$ in a neighborhood of $z_{0}$. {{To justify the final identity and to see that the power series converges uniformly, we can use \cite[Proposition 3.1]{Kalton-Mitrea98} to write 
\begin{align*}
&\sum_{k=0}^{\infty}|z-z_{0}|^{k}\sum_{l=0}^{\infty}\|(z-z_{0})^{l}y_{k,l}\|_{Y_{0}+Y_{1}}\\
&\lesssim \sum_{k=0}^{\infty}\sum_{l=0}^{\infty}\frac{(m+l)!}{l!}|z-z_{0}|^{k+l}\sup_{|z'-z_{0}|<\veps}\|T_{z'}x_{k}\|_{Y_{0}+Y_{1}}\\
&\lesssim \sum_{k=0}^{\infty}\sum_{l=0}^{\infty}\frac{(m+l)!}{l!}|z-z_{0}|^{k+l}\|x_{k}\|_{X_{0}+X_{1}}\lesssim \sum_{k=0}^{\infty}|z-z_{0}|^{k}\|x_{k}\|_{X_{0}+X_{1}}
\end{align*}
for some $m\in\N$ and $\veps>0$ independent of $k$ and $l$, and for all $z\in\St$ with $|z-z_{0}|<\veps$. The final series in turn converges locally uniformly because $f$ is analytic, and we see that $z\mapsto T_{z}f(z)$ is indeed analytic on $\St$.}}

Next, by \eqref{it:Steininter3}, $T_{j+it}f(j+it)\in Y_{j}$ for all $t\in\R$ and $j\in\{0,1\}$, with
\begin{equation}\label{eq:Steininter4}
\begin{aligned}
&\sup_{t\in\R}\|T_{j+it}f(j+it)\|_{Y_{j}}\leq K_{j}\sup_{t\in\R}\|f(j+it)\|_{X_{j}}.
\end{aligned}
\end{equation}
And, combining \eqref{it:Steininter3} and \eqref{it:Steininter4}, it follows that $t\mapsto T_{j+it}f(j+it)$ is a continuous $Y_{j}$-valued map. In turn, \eqref{eq:Steininter3} and \eqref{eq:Steininter4} imply \eqref{eq:Steininter1}. 

Finally, if $Y_{0}+Y_{1}$ is analytically convex, then
\[
\sup_{z\in\overline{\St}}\|T_{z}f(z)\|_{Y_{0}+Y_{1}}=\sup_{z\in\St}\|T_{z}f(z)\|_{Y_{0}+Y_{1}}\lesssim \sup\{\|f(z)\|_{Y_{0}+Y_{1}}\mid \Real(z)\in\{0,1\}\}
\]
for each $f\in\F_{X_{0},X_{1}}$, after which \eqref{eq:Steininter4} concludes the proof. 
\end{proof}

We also use the notion of analytic convexity in the proof of the following result, stated in the main text as Proposition \ref{prop:tentint}.

\begin{proposition}\label{prop:tentintapp}
Let $p_{0},p_{1}\in(0,\infty]$ be such that $(p_{0},p_{1})\neq (\infty,\infty)$, and let $p\in(0,\infty)$, $s_{0},s_{1},s\in\R$ and $\theta\in(0,1)$ be such that $\frac{1}{p}=\frac{1-\theta}{p_{0}}+\frac{\theta}{p_{1}}$ and $s=(1-\theta)s_{0}+\theta s_{1}$. Then
\[
[T^{p_{0}}_{s_{0}}(\Sp),T^{p_{1}}_{s_{1}}(\Sp)]_{\theta}=T^{p}_{s}(\Sp),
\]
with equivalent quasi-norms.
\end{proposition}
\begin{proof}

We mostly follow the approach from the Euclidean analogue of this result in \cite[Lemma 8.23]{HoMaMc11}, which involves some additional restrictions on the parameters. 

It is straightforward to check that $T^{q}_{r}(\Sp)$ is a quasi-Banach function space over $\Spp$, in the sense of \cite[Section 3]{Egert-Kosmala24}, for all $0<q\leq\infty$ and $r\in\R$. Moreover, it was observed in \cite{HaToVi91} that $T^{q}_{r}(\Sp)$, for $q<\infty$, is isometrically isomorphic to a closed subspace of $L^{q}(\Sp;L^{2}(\Spp,w_{r}))$, where $w_{r}$ is a suitable weight on $\Spp$ (see also \cite[page 120]{Amenta18}). As such, $T^{q}_{r}(\Sp)$ is both separable and, by \cite[Proposition 7.5 and Lemma 7.6]{KaMaMi07}, analytically convex. On the other hand, $T^{\infty}_{r}(\Sp)$ is a Banach space, and as such also analytically convex.

Now \cite[Theorem 1]{Egert-Kosmala24} implies that  $T^{p_{0}}_{s_{0}}(\Sp)+T^{p_{1}}_{s_{1}}(\Sp)$ is analytically convex as well, and that
\begin{equation}\label{eq:Calderon}
[T^{p_{0}}_{s_{0}}(\Sp),T^{p_{1}}_{s_{1}}(\Sp)]_{\theta}=T^{p_{0}}_{s_{0}}(\Sp)^{1-\theta}T^{p_{1}}_{s_{1}}(\Sp)^{\theta}
\end{equation}
for each $\theta\in(0,1)$. Here  $T^{p_{0}}_{s_{0}}(\Sp)^{1-\theta}T^{p_{1}}_{s_{1}}(\Sp)^{\theta}$ is the \emph{Calder\'{o}n product}, consisting of all measurable $H:\Spp\to\C$ for which there exist $F\in T^{p_{0}}_{s_{0}}(\Sp)$ and  $G\in T^{p_{1}}_{s_{1}}(\Sp)$ such that $|H|\leq |F|^{1-\theta}|G|^{\theta}$, endowed with the natural associated infimum quasi-norm. 

Next, for $0<q<\infty$, $r\in\R$ and $F:\Spp\to\C$ measurable, set 
\[
\A_{r,q} F(x,\w):=\Big(\int_{0}^{\infty}\fint_{B_{\sqrt{\sigma}}(x,\w)}|{\sigma}^{-r}F(y,\nu,\sigma)|^{q}\ud y\ud \nu\frac{\ud \sigma}{\sigma}\Big)^{1/q}
\]
and
\[
\mathcal{C}_{r,q}F(x,\w):=\sup_{B}\Big(\frac{1}{|B|}\int_{T(B)}|\sigma^{-r}F(y,\nu,\sigma)|^{q}\ud y\ud \nu\frac{\ud \sigma}{\sigma}\Big)^{1/q},
\]
where the supremum is taken over all balls $B\subseteq \Sp$ containing $(x,\w)$, with the natural modification for $q=\infty$. 
For $0<t<\infty$, let $T^{t,q}_{r}(\Sp)$ consist of all measurable $F:\Spp\to\C$ such that 
\[
\|F\|_{T^{t,q}_{r}(\Sp)}:=\|\A_{r,q}F\|_{L^{t}(\Sp)}<\infty,
\]
and let $T^{\infty,q}_{r}(\Sp)$ consist of those measurable $F:\Spp\to\C$ such that
\[
\|F\|_{T^{\infty,q}_{r}(\Sp)}:=\|C_{r,q}F\|_{L^{\infty}(\Sp)}<\infty.
\]
It then follows from \cite{Cohn-Verbitsky00} that 
\begin{equation}\label{eq:tentproduct}
T^{t_{0},q_{0}}_{0}(\Sp)T^{t_{1},q_{1}}_{0}(\Sp)=T^{t,q}_{0}(\Sp)
\end{equation}
for all $t_{0},t_{1},q_{0},q_{1},t,q\in(0,\infty]$ such that $\frac{1}{t}=\frac{1}{t_{0}}+\frac{1}{t_{1}}$ and $\frac{1}{q}=\frac{1}{q_{0}}+\frac{1}{q_{1}}$. Here the product in \eqref{eq:tentproduct} is defined analogously to the Calder\'{o}n product. In fact, \cite{Cohn-Verbitsky00} considers the Euclidean setting, but the same method yields \eqref{eq:tentproduct}, at least if one avoids the use of the Poisson kernel in the proof of \cite[Lemma 2.1]{Cohn-Verbitsky00} and only relies on the operator $P_{0}$ considered there.

Now, it easily follows from \eqref{eq:tentproduct} that
\[
T^{t_{0},q_{0}}_{r_{0}}(\Sp)T^{t_{1},q_{1}}_{r_{1}}(\Sp)=T^{t,q}_{r}(\Sp)
\]
for all $r,r_{0},r_{1}\in\R$ with $r=r_{0}+r_{1}$, and one has $T^{t,q}_{r}(\Sp)^{u}=T^{t/u,q/u}_{ru}(\Sp)$ for each $0<u<\infty$. Applying this with $q_{0}=q_{1}=2$ and recalling \eqref{eq:Calderon}, we see that
\[
[T^{p_{0}}_{s_{0}}(\Sp),T^{p_{1}}_{s_{1}}(\Sp)]_{\theta}=T^{p_{0}}_{s_{0}}(\Sp)^{1-\theta}T^{p_{1}}_{s_{1}}(\Sp)^{\theta}=T^{p}_{s}(\Sp),
\] 
as required.
\end{proof}


\bibliographystyle{plain}
\bibliography{bibliography}
\end{document}